\documentclass{amsart}

\usepackage{amscd}
\usepackage{amssymb}

\usepackage[T1]{fontenc}
\usepackage[latin1]{inputenc}
\usepackage{hyperref}
\usepackage[arrow,curve,matrix]{xy}
\usepackage{times}
\usepackage{graphicx}
\usepackage{color}
\usepackage{flafter}
\usepackage{placeins}
\usepackage{enumitem}

\DeclareFontFamily{OMS}{rsfs}{\skewchar\font'60}
\DeclareFontShape{OMS}{rsfs}{m}{n}{<-5>rsfs5 <5-7>rsfs7 <7->rsfs10 }{}
\DeclareSymbolFont{rsfs}{OMS}{rsfs}{m}{n}
\DeclareSymbolFontAlphabet{\scr}{rsfs}

\renewcommand{\P}{\mathbb{P}} \newcommand{\C}{\mathbb{C}}
\renewcommand{\O}{\sO}
\newcommand{\Q}{\mathbb{Q}}

\newcommand\resto[1]{\hbox{\hbox{$\big\vert{}_{_{#1}}$}}}

\newcommand{\blc}{boundary-lc}

\newcommand{\acts}{\mbox{\raisebox{0.26ex} {\tiny {$\bullet$}}} }

\newcommand{\sA}{\scr{A}}

\newcommand{\sC}{\scr{C}}

\newcommand{\sF}{\scr{F}}
\newcommand{\sG}{\scr{G}}

\newcommand{\sI}{\scr{I}}
\newcommand{\sJ}{\scr{J}}

\newcommand{\sM}{\scr{M}}
\newcommand{\sN}{\scr{N}}
\newcommand{\sO}{\scr{O}}

\newcommand{\sT}{\scr{T}}

\newcommand{\bN}{\mathbb{N}}

\newcommand{\bP}{\mathbb{P}}
\newcommand{\bQ}{\mathbb{Q}}

\newcommand{\Top}{{\sf T}}

\DeclareMathOperator{\Exc}{Exc}

\DeclareMathOperator{\codim}{codim}

\DeclareMathOperator{\Ext}{Ext}

\DeclareMathOperator{\lc}{lc}

\DeclareMathOperator{\reg}{reg}
\DeclareMathOperator{\red}{red}
\DeclareMathOperator{\sing}{sing}

\DeclareMathOperator{\Sym}{Sym}
\DeclareMathOperator{\supp}{supp}

\DeclareMathOperator{\Id}{{Id}}

\newcommand{\into}{\hookrightarrow}

\newcommand{\wt}{\widetilde}

\newcommand{\wtilde}{\widetilde}

\newcounter{thisthm}
\newcommand{\ilabel}[1]{\newcounter{#1}\setcounter{thisthm}{\value{thm}}\setcounter{#1}{\value{enumi}}}
\newcommand{\iref}[1]{(\thesection.\the\value{thisthm}.\the\value{#1})}

\theoremstyle{plain}    
\newtheorem{thm}{Theorem}[section]
\newtheorem{defn}[thm]{Definition}

\numberwithin{equation}{thm}
\numberwithin{figure}{section}
\theoremstyle{plain}    
\newtheorem{cor}[thm]{Corollary}
\newtheorem{lem}[thm]{Lemma}

\newtheorem{fact}[thm]{Fact}

\theoremstyle{plain}    
\newtheorem{prop}[thm]{Proposition}
\newtheorem{proclaim-special}[thm]{\specialthmname}

\theoremstyle{remark}
\newtheorem{rem}[thm]{Remark}

\newtheorem{warning}[thm]{Warning}
\newtheorem{subrem}[equation]{Remark}
\newtheorem*{claim*}{Claim} 
\newtheorem{notation}[thm]{Notation}
 
\newtheorem{example}[thm]{Example}
\newtheorem{subexample}[equation]{Example}

\newtheoremstyle{bozont-remark}{3pt}{3pt}%
     {}
     {}
     {\it}
     {.}
     {.5em}
     {\thmname{#1}\thmnumber{ #2}: \thmnote{\sc #3}}
\theoremstyle{bozont-remark}

\def\factor#1.#2.{\left. \raise 2pt\hbox{$#1$} \right/\hskip -2pt\raise
  -2pt\hbox{$#2$}} 
\setenumerate[1]{leftmargin=*,parsep=0em,itemsep=0.125em,topsep=0.125em}
\newlength{\swidth}
\newenvironment{enumerate-p}{
  \begin{enumerate}}
  {\setcounter{equation}{\value{enumi}}\end{enumerate}}

\date{\today}

\author{Daniel Greb}
\author{Stefan Kebekus}
\author{S\'andor J Kov\'acs}

\thanks{Stefan Kebekus was supported in part by the DFG-Forschergruppe
  ``Classification of Algebraic Surfaces and Compact Complex Manifolds''.  S\'andor
  Kov\'acs was supported in part by NSF Grant DMS-0554697 and the Craig McKibben and
  Sarah Merner Endowed Professorship in Mathematics.}

\address{Daniel Greb, Albert-Ludwigs-Universit\"at Freiburg, Mathematisches
  Institut, Eckerstrasse 1, 79104 Freiburg im Breisgau, Germany}
\email{\href{mailto:daniel.greb@math.uni-freiburg.de}{daniel.greb@math.uni-freiburg.de}}
\urladdr{\href{http://home.mathematik.uni-freiburg.de/dgreb}{http://home.mathematik.uni-freiburg.de/dgreb}}
  
\address{Stefan Kebekus, Albert-Ludwigs-Universit\"at Freiburg, Mathematisches
  Institut, Eckerstrasse 1, 79104 Freiburg im Breisgau, Germany}
\email{\href{mailto:stefan.kebekus@math.uni-freiburg.de}{stefan.kebekus@math.uni-freiburg.de}}
\urladdr{\href{http://home.mathematik.uni-freiburg.de/kebekus}{http://home.mathematik.uni-freiburg.de/kebekus}}

\address{S\'andor Kov\'acs, University of Washington,
  Department of Mathematics, Box 354350, Seattle, WA 98195, U.S.A.}
\email{\href{mailto:kovacs@math.washington.edu}{kovacs@math.washington.edu}}
\urladdr{\href{http://www.math.washington.edu/~kovacs}{http://www.math.washington.edu/$\sim$kovacs}}

\newcommand{\PreprintAndPublication}[2]{#1} 


\title[Extension theorems on log canonical varieties]{Extension
  theorems for differential forms and Bogomolov-Sommese vanishing on
  log canonical varieties}

\begin{document}

\begin{abstract}
  Given a normal variety $Z$, a $p$-form $\sigma$ defined on the
  smooth locus of $Z$, and a resolution of singularities $\pi :
  \wtilde Z \to Z$, we study the problem of extending the pull-back
  $\pi^*(\sigma)$ over the $\pi$-exceptional set $E \subset \wtilde
  Z$.

  For log canonical pairs and for certain values of $p$, we show that
  an extension always exists, possibly with logarithmic poles along
  $E$. As a corollary, it is shown that sheaves of reflexive
  differentials enjoy good pull-back properties. A natural
  generalization of the well-known Bogomolov-Sommese vanishing theorem
  to log canonical threefold pairs follows.
\end{abstract}

\maketitle
\tableofcontents

\section{Introduction and statement of main result}

\subsection{Introduction}

Let $Z$ be a normal projective variety and $\sigma \in H^0 \bigl( Z,\,
\Omega^{[p]}_Z \bigr)$ a $p$-form which is defined away from the
singularities. A natural question to ask is: If $\pi: \wtilde Z \to Z$
is a resolution of singularities, can one extend $\pi^*(\sigma)$ as a
differential form to all of $\wtilde Z$, perhaps allowing logarithmic
poles along the $\pi$-exceptional set?

If $p = \dim Z$ and if the pair $(Z, \emptyset)$ is log canonical, the
answer is ``yes'', almost by definition. For other values of $p$, the
problem has been studied by Hodge-theoretic methods ---see the papers
of Steenbrink \cite{Steenbrink85}, Steenbrink-van Straten \cite{SS85},
Flenner \cite{Flenner88} and the references therein. In a nutshell,
the answer is ``yes'' if the codimension of the singular set is large.

In this paper, we consider logarithmic varieties with log canonical
singularities. We show that for these varieties and certain values of
$p$, the answer is ``yes'', irrespective of the codimension of the
singular set.

As a corollary, we show that sheaves of reflexive differentials enjoy
good pull-back properties and prove a version of the well-known
Bogomolov-Sommese vanishing theorem for log canonical threefold pairs.

\subsection{Main results}\label{sec:mainresults}

The following is the main result of this paper. In essence, it asserts
that a (logarithmic) $p$-form defined away from the singular set of a
log canonical threefold pair gives rise to $p$-forms on any log
resolution.

\begin{thm}[Extension theorem for log canonical pairs]\label{thm:extension-lc} 
  Let $Z$ be a normal variety of dimension $n$ and $\Delta \subset Z$ a reduced
  divisor such that the pair $(Z, \Delta)$ is log canonical.  Let $\pi: \wtilde Z \to
  Z$ be a log resolution, and set
  $$
  \wtilde \Delta_{\lc} := \text{largest reduced divisor contained in }
  \pi^{-1}\bigl(\Delta \cup \text{non-klt locus of }(Z, \Delta)\bigr),
  $$
  where the non-klt locus is the minimal closed subset $W \subset Z$ such that
  that pair $(Z, \Delta)$ is klt away from $W$. If $p \in \{n, n-1, 1\}$, then
  the sheaf $\pi_* \Omega^p_{\wtilde Z}(\log \wtilde \Delta_{\lc})$ is
  reflexive.
\end{thm}

\begin{subrem}
  Logarithmic differentials are introduced and discussed in
  \cite[Chapt.~11c]{Iitaka82} or \cite[Chap.~3]{Deligne70}. The notion of log
  resolution is recalled in Definition~\ref{def:everythinglog2b} below. We
  refer the reader to \cite[Sect.\ 2.3]{KM98} for the definition of log
  canonical and klt singularities.
\end{subrem}
\begin{subrem}
  Since the coefficients of its components are equal to $1$
  (cf.~Definition~\ref{def:everythinglog1}), the boundary divisor $\Delta$ is
  contained in the non-klt locus of $(X, \Delta)$. We have nevertheless chosen
  to explicitly include it in the definition of $\wtilde \Delta_{\lc}$ for
  reasons of clarity.
\end{subrem}

The name ``extension theorem'' is justified by the following remark.

\begin{rem}\label{rem:reflexiveextension}
  Theorem~\ref{thm:extension-lc} asserts precisely that for any open
  set $U \subset Z$ and any number $p \in \{n, n-1, 1\}$, the
  restriction morphism
  \begin{equation}\label{eq:rest}
    H^0\bigl(\pi^{-1}(U),\, \Omega^p_{\wtilde Z}(\log \wtilde \Delta_{\lc})\bigr) \to
    H^0\bigl(\pi^{-1}(U) \setminus \Exc(\pi),\, \Omega^p_{\wtilde Z}(\log \wtilde
    \Delta_{\lc})\bigr) 
  \end{equation}
  is surjective, where $\Exc(\pi) \subset \wtilde Z$ denotes the
  $\pi$-exceptional set,
\end{rem}

\begin{rem}
  After this paper appeared in preprint form we learned that more general results had
  been claimed in \cite[Thms.~4.9 and 4.11]{MR1971155}. However, in discussions with
  A.~Langer
  we found that the proof of \cite[Thm.~4.9]{MR1971155}
  contains a gap that at present has still not been filled: In the last paragraph of
  the proof, it is not clear that the prerequisites of \cite[Lem.~4.8]{MR1971155} are
  satisfied. 
  For a special case of the statement for surfaces, see \cite[Thm.~4.2]{MR1853454}.
\end{rem}

For an application of Theorem~\ref{thm:extension-lc}, recall the well-known
Bogomolov-Sommese vanishing theorem for snc pairs, cf.~\cite[Cor.~6.9]{EV92}: If $Z$
is a smooth projective variety, $\Delta \subset Z$ a divisor with simple normal
crossings and $\sA \subset \Omega^p_Z(\log \Delta)$ an invertible subsheaf, then the
Kodaira-Iitaka dimension of $\sA$ is not larger than $p$, i.e., $\kappa(\sA) \leq p$.
As a corollary to Theorem~\ref{thm:extension-lc}, we will show in
Section~\ref{sec:bogomolov} that a similar result holds for threefold pairs with log
canonical singularities. We refer to Definition~\ref{def:kappa1} for the definition
of the Kodaira-Iitaka dimension for sheaves that are not necessarily locally free.

\begin{thm}[Bogomolov-Sommese vanishing for log canonical threefolds and
  surfaces]\label{thm:Bvanishing} 
  Let $Z$ be a normal variety of dimension $\dim Z \leq 3$ and let $\Delta \subset Z$
  be a reduced divisor such that the pair $(Z, \Delta)$ is log canonical. Let $\sA
  \subset \Omega^{[p]}_Z(\log \Delta)$ be a reflexive subsheaf of rank one. If $\sA$
  is $\Q$-Cartier, then $\kappa(\sA) \leq p$.
\end{thm}

In fact, a stronger result holds---see Theorem~\ref{thm:Bvanishing2}
on page~\pageref{thm:Bvanishing2}.

\subsection{Outline of the paper}

We introduce notation and recall standard facts in Section~\ref{sec:notation}. In
Section~\ref{sec:ext-ext} we prepare for the proof of Theorem~\ref{thm:extension-lc}
by showing how extension properties of a given space $Z$ can often be deduced from
extension properties of finite covers of $Z$. This already gives extension results
for an important class of surface singularities that appears naturally within the
minimal model program.  Because of their importance in applications, we briefly
discuss these singularities in Subsection~\ref{ssec:dlc}.

Theorem~\ref{thm:extension-lc} is shown in Sections~\ref{sec:ext-n}--\ref{sec:ext-1}
for $n$-forms, $(n-1)$-forms and $1$-forms, respectively. The proof of the extension
result for $(n-1)$-forms relies on universal properties of the functorial resolution
of singularities and on liftings of local group actions.  The extension for $1$-forms
is shown using results of Steenbrink and Namikawa that are Hodge-theoretic in nature.

Section~\ref{sec:bogomolov} discusses pull-back properties of sheaves of
differentials and gives a proof of the Bogomolov-Sommese vanishing theorem for log
canonical threefolds and surfaces, Theorem~\ref{thm:Bvanishing}. For the reader's
convenience, an appendix recalling the variant of Hartshorne's formal duality theorem
for cohomology with supports that is required in our context is included,
cf.~Section~\ref{ssec:extoverdiscr}.

\subsection{Acknowledgments}

We would like to thank Thomas Peternell, Duco van~Straten and Chengyang Xu for
numerous discussions that motivated the problem and helped to improve this paper. We
would also like to thank Joseph Steenbrink for kindly answering our questions by
e-mail.

\part{TOOLS}

\section{Notation and standard facts}
\label{sec:notation}

\subsection{Reflexive tensor operations}

When dealing with sheaves that are not necessarily locally free, we
frequently use square brackets to indicate taking the reflexive hull.

\begin{notation}
  Let $Z$ be a normal variety and $\sA$ a coherent sheaf of
  $\O_Z$-modules. Let $n\in \bN$ and set $\sA^{[n]} := \otimes^{[n]}
  \sA := (\sA^{\otimes n})^{**}$, $\Sym^{[n]} \sA := (\Sym^n
  \sA)^{**}$, etc. Likewise, for a morphism $\gamma : X \to Z$ of
  normal varieties, set $\gamma^{[*]}\sA := (\gamma^*\sA)^{**}$. If
  $\sA$ is reflexive of rank one, we say that $\sA$ is $\Q$-Cartier if
  there exists an $n\in\bN$ such that $\sA^{[n]}$ is invertible.
\end{notation}

In the sequel, we will frequently state and prove results that hold
for the sheaf of differentials $\Omega^{[1]}_Z$, the reflexive hull of
its symmetric products, exterior products, tensor products, or any
combination of these tensor operations. The following shorthand
notation is therefore useful.

\begin{notation}
  A reflexive tensor operation is any combination of the reflexive
  tensor product $\otimes^{[k]}$, the symmetric product $\Sym^{[l]}$
  or the exterior product $\bigwedge^{[m]}$. If $\Top$ is a tensor
  operation, such as $\Top = \otimes^{[2]}\Sym^{[3]}$, and $\sF$ is a
  sheaf of $\sO_Z$-modules on a scheme $Z$, we often write $\Top \sF$
  instead of $\otimes^{[2]}_{\sO_Z} \Sym^{[3]}_{\sO_Z} \sF$.
\end{notation}

We will be working with the Kodaira-Iitaka dimension of reflexive
sheaves on normal spaces. Since this is perhaps not quite standard, we
recall the definition here.

\begin{defn}[Kodaira-Iitaka dimension]\label{def:kappa1}
  Let $Z$ be a normal projective variety and $\sA$ a reflexive sheaf
  of rank one on $Z$.  If $h^0\bigl(Z,\, \sA^{[n]}\bigr) = 0$ for all
  $n \in \mathbb N$, then we say that $\sA$ has Kodaira-Iitaka
  dimension $\kappa(\sA) := -\infty$.  Otherwise, set $$M := \bigl\{
  n\in \mathbb N \,|\, h^0\bigl(Z,\, \sA^{[n]}\bigr)>0\bigr\}.$$
  Recall that the restriction of $\sA$ to the smooth locus of $Z$ is
  locally free and consider the rational mapping
  $$
  \phi_n : Z \dasharrow \mathbb P\bigl(H^0\bigl(Z,\,
  \sA^{[n]}\bigr)^*\bigr) \text{ for each } n \in M.
  $$
  The Kodaira-Iitaka dimension of $\sA$ is then defined as
  $$
  \kappa(\sA) := \max_{n \in M} \bigl(\dim
  \overline{\phi_n(Z)}\bigr).
  $$
\end{defn}

\subsection{Logarithmic pairs and the extension theorem}

For the reader's convenience, we recall a few definitions of
logarithmic geometry. Although not quite standard, the following
notion of a morphism of logarithmic pairs is useful for our purposes.

\begin{defn}[Logarithmic pair]\label{def:everythinglog1}
  A \emph{logarithmic pair} $(Z,\Delta)$ consists of a normal variety
  or complex space $Z$ and a reduced, but not necessarily irreducible
  Weil divisor $\Delta \subset Z$.  A \emph{morphism of logarithmic
    pairs} $\gamma: (\wtilde Z, \wtilde \Delta) \to (Z,\Delta)$ is a
  morphism $\gamma: \wtilde Z \to Z$ such that $\gamma^{-1}(\Delta) =
  \wtilde \Delta$ set-theoretically.
\end{defn}

\begin{defn}[Snc pairs]\label{def:everythinglog2a}
  Let $(Z, \Delta)$ be a logarithmic pair, and $z \in Z$ a point. We say that $(Z,
  \Delta)$ is \emph{snc at $z$}, if there exists a Zariski-open neighborhood $U$ of
  $z$ such that $U$ is smooth and $\Delta \cap U$ has only simple normal crossings.
  The pair $(Z, \Delta)$ is \emph{snc} if it is snc at all $z\in Z$.

  Given a logarithmic pair $(Z,\Delta)$, let $(Z,\Delta)_{\reg}$ be the maximal open
  set of $Z$ where $(Z,\Delta)$ is snc, and let $(Z,\Delta)_{\sing}$ be its
  complement, with the induced reduced subscheme structure.
\end{defn}

\begin{subrem} If a logarithmic pair $(Z, \Delta)$ is snc at a point $z$,
    this implies that all components of $\Delta$ are smooth at $z$. Without the
    condition that $U$ is Zariski-open this would no longer be true, and
    Definition~\ref{def:everythinglog2a} would define normal crossing pairs rather
    than pairs with simple normal crossing.
\end{subrem}

\begin{defn}[Log resolution]\label{def:everythinglog2b}
  A \emph{log resolution} of $(Z, \Delta)$ is a birational morphism of pairs $\pi:
  (\wtilde Z, \wtilde \Delta) \to (Z,\Delta)$ such that the $\pi$-exceptional set
  $\Exc(\pi)$ is of pure codimension one, such that $\bigl(\wtilde Z,\, \supp(\wtilde
  \Delta\cup \Exc(\pi)) \bigr)$ is snc, and such that $\pi$ is isomorphic over $(Z,
  \Delta)_{\reg}$.
\end{defn}

The following definitions will be helpful in the proof of
Theorem~\ref{thm:extension-lc} and its corollaries.

\begin{notation}
  If $(Z, \Delta)$ is a logarithmic pair, and $\Top$ a reflexive
  tensor operation, the sheaf $\Top\Omega^1_Z(\log \Delta)$ will be
  called the sheaf of $\Top$-forms.
\end{notation}

\begin{defn}[Extension theorem]\label{def:extensionthmholds}
  If $(Z, \Delta)$ is a logarithmic pair, and $\Top$ a reflexive
  tensor operation, we say that \emph{the extension theorem holds for
    $\Top$-forms on $(Z, \Delta)$}, if the following holds: Let $\pi:
  (\wtilde Z, \wtilde \Delta) \to (Z, \Delta)$ be a log resolution and
  $E_\Delta$ the union of all $\pi$-exceptional components not
  contained in $\wtilde \Delta$. Then the push-forward sheaf
  $$
  \pi_* \Top \Omega^1_{\wtilde Z}(\log( \wtilde \Delta + E_\Delta))
  $$
  is reflexive. Equivalently, the restriction morphism
  \begin{equation}
    H^0 \bigl( \pi^{-1}(U),\, \Top \Omega^1_{\wtilde Z}(\log (\wtilde \Delta +
    E_\Delta)) \bigr) \to H^0 
    \bigl( \pi^{-1}(U) \setminus \Exc(\pi),\, \Top \Omega^1_{\wtilde Z}(\log \wtilde \Delta)
    \bigr)
  \end{equation}
  is surjective for any open set $U \subseteq Z$.
\end{defn}

\subsection{Pull-back properties of logarithmic and regular differentials}

Morphisms of snc pairs give rise to pull-back morphisms of logarithmic
differentials. In this section, we briefly recall the standard fact
that the pull-back morphism associated with a finite map is isomorphic
if the branch locus is contained in the boundary. We refer to
\cite[Chap.~11]{Iitaka82} for details.

\begin{fact}\label{fact:pullback}
  Let $\gamma: (\wtilde Z, \wtilde \Delta) \to (Z,\Delta)$ be a
  morphism of snc pairs, $U\subseteq Z$ an open set and
  $\wtilde U=\gamma^{-1}(U)$. Then there exists a natural pull-back
  map of forms $$\gamma^* :H^0\bigl(U,\, \Omega^1_Z(\log \Delta)\bigr)
  \to H^0\bigl(\wtilde U,\, \Omega^1_{\wtilde Z}(\log \wtilde
  \Delta)\bigr),$$ and an associated sheaf morphism
$$d\gamma:
  \gamma^*
\Omega^1_Z(\log \Delta)
 \to \Omega^1_{\wtilde Z}(\log
  \wtilde \Delta).$$ If $\gamma$ is finite and unramified over $Z
  \setminus \Delta$, then $d\gamma$ is an isomorphism.  \qed
\end{fact}

\begin{rem}
  If $\Top$ is any reflexive tensor operation, then the pull-back
  morphism also gives a pull-back of $\Top$-forms, $\gamma^* :
  H^0\bigl(Z,\,\Top \Omega^1_Z(\log \Delta)\bigr) \to H^0\bigl(\wtilde
  Z,\, \Top \Omega^1_{\wtilde Z}(\log \wtilde \Delta)\bigr)$, that
  obviously extends to a pull-back of rational $\Top$-forms.
\end{rem}

We state one immediate consequence for future reference. The following
notation is useful in the formulation.

\begin{notation}
  Let $X$ be a normal variety, $\Gamma \subset X$ a reduced Weil divisor and
  $\sF$ a reflexive coherent sheaf of $\sO_X$-modules. We will often consider
  sections of $\sF|_{X\setminus \Gamma}$. Equivalently, we consider rational
  sections of $\sF$ with poles of arbitrary order along $\Gamma$, and let
  $\sF(*\Gamma)$ be the associated sheaf of these sections on $X$.  More
  precisely, we define
  $$
  \sF(*\Gamma) := \lim_{\overset{\longrightarrow}{m}}\left(\bigl(\sF\otimes
      \sO_X(m\cdot \Gamma)\bigr)^{**} \right).
  $$
  With this notation we have $H^0 \bigl( X,\, \sF(*\Gamma) \bigr) =
  H^0\bigl( X\setminus \Gamma,\, \sF \bigr)$.
\end{notation}

\begin{cor}\label{cor:pb2}
  Under the conditions of Fact~\ref{fact:pullback}, let $\Top$ be any
  reflexive tensor operation and assume that $\gamma$ is a finite
  morphism. Let $\Gamma \subset Z$ be a reduced divisor and $\sigma \in
  H^0\bigl(Z,\, \Top \Omega^1_Z(\log \Delta)(*\Gamma)\bigr)$ a $\Top$-form
  that might have poles along $\Gamma$.
  \begin{enumerate}
  \item\label{il:pbcrit1} If $\gamma$ is unramified over $Z \setminus \Delta$,
    then the form $\sigma$ has only logarithmic poles along $\Gamma$ if and
    only if $\gamma^*(\sigma)$ has only logarithmic poles along $\supp
    \bigl(\gamma^{-1}(\Gamma) \bigr)$, i.e.,
    $$
    \sigma \in H^0\bigl(Z,\, \Top \Omega^1_Z(\log \Delta)\bigr) \quad
    \Leftrightarrow \quad \gamma^*(\sigma) \in H^0\bigl( \wtilde Z,\,
    \Top \Omega^1_{\wtilde Z}(\log \wtilde \Delta) \bigr).
    $$
  \item\label{il:pbcrit2} If\, $\Top = \bigwedge^{[p]}$, then $\sigma$
    is a regular form if and only if $\gamma^*(\sigma)$ is regular,
    i.e.,
    $$
    \sigma \in H^0\bigl(Z,\, \Omega^p_Z\bigr) \quad \Leftrightarrow
    \quad \gamma^*(\sigma) \in H^0\bigl( \wtilde Z,\,
    \Omega^p_{\wtilde Z} \bigr).
    $$
  \end{enumerate}
\end{cor}
\begin{proof}
  Assertion~(\ref{cor:pb2}.\ref{il:pbcrit1}) follows immediately from
  Fact~\ref{fact:pullback}. The proof
  of~(\ref{cor:pb2}.\ref{il:pbcrit2}) is left to the reader.
\end{proof}

\subsection{Comparing log resolutions}

Reflexivity of the push-forward of sheaves of differentials from an arbitrary
birational model of a given pair can often be concluded if we know the reflexivity of
the push-forward from a particular log resolution. This is summarized in the
following elementary lemma.

\begin{lem}\label{lem:oneresolutionsuffices}
  Let $(Z, \Delta)$ be a logarithmic pair and $W \subset Z$ a subvariety. For
  $i \in \{1,2\}$, let $\pi_i: (Z_i, \Delta_i) \to (Z, \Delta)$ be a
  birational morphism of logarithmic pairs and
  $$
  \Gamma_i := \text{largest reduced divisor contained in } \pi_i^{-1}(\Delta
  \cup W).
  $$
  If $\Top$ is a reflexive tensor operation, $(Z_2, \Gamma_2)$ is snc and
  $(\pi_2)_* \Top \Omega_{Z_2}^1 (\log \Gamma_2)$ is reflexive, then
  $(\pi_1)_* \Top \Omega_{Z_1}^1 (\log \Gamma_1)$ is reflexive as well.
\end{lem}

\begin{subrem}
  In the setup of Lemma~\ref{lem:oneresolutionsuffices}, the sheaves
  $(\pi_1)_* \Top \Omega_{Z_1}^1 (\log \Gamma_1)$ and $(\pi_2)_* \Top
  \Omega_{Z_2}^1 (\log \Gamma_2)$ are isomorphic away from a set of
  codimension at least two. If the sheaves are reflexive, this implies that
  they are in fact isomorphic.
\end{subrem}

\begin{proof}[Proof of Lemma~\ref{lem:oneresolutionsuffices}]
  Choose an snc logarithmic pair $(\wtilde Z, \wtilde \Delta)$, together with
  birational morphisms of pairs $\varphi_i: (\wtilde Z, \wtilde \Delta) \to
  (Z_i, \Delta_i)$ such that $\wtilde \Gamma_2 := \supp \bigl( \varphi_2^{-1}
  (\Gamma_2) \bigr)$ is a divisor with snc support and such that the following
  diagram commutes:
  $$
  \begin{xymatrix}{
      (\wtilde Z, \wtilde \Delta) \ar^{\varphi_2}[r]\ar_{\varphi_1}[d]  &   (Z_2, \Delta_2) \ar^{\pi_2}[d] \\
      (Z_1, \Delta_1)\ar_{\pi_1}[r] & (Z, \Delta) & }
  \end{xymatrix}
  $$
  Let $U \subseteq Z$ be open and $\sigma \in H^0\bigl(U,\, \Top
  \Omega_Z^1 (\log \Delta)\bigr)$ a $\Top$-form on $U$. For
  convenience, set $\psi := \pi_1 \circ \varphi_1= \pi_2 \circ
  \varphi_2$ and denote the preimages of $U$ on $Z_1$, $Z_2$, and
  $\wtilde Z$ by $U_1$, $U_2$, and $\wtilde U$ respectively.

  By assumption, $\pi_2^* (\sigma)$ extends to a $\Top$-form on $(Z_2,
  \Gamma_2)$ without poles along the exceptional set $\Exc(\pi_2)$, i.e.,
  $\pi_2^* (\sigma) \in H^0 \bigl(U_2,\, \Top \Omega_{Z_2}^1 (\log \Gamma_2 )
  \bigr)$.  If we set
  $$
  \wtilde \Gamma := \text{largest reduced divisor contained in }
  \psi^{-1}(\Delta \cup W),
  $$
  then $\wtilde \Gamma$ contains $\wtilde \Gamma_2$ and
  Fact~\ref{fact:pullback} implies that $\psi^*(\sigma)$ extends to a
  $\Top$-form on $\bigl(\wtilde U, \wtilde \Gamma_2 \bigr)$. In particular,
  $$
  \psi^* (\sigma) \in H^0 \bigl(\wtilde U,\, \Top \Omega_{\wtilde Z}^1 (\log
  \wtilde \Gamma_2) \bigr) \subseteq H^0 \bigl(\wtilde U,\, \Top
  \Omega_{\wtilde Z}^1 (\log \wtilde \Gamma) \bigr).
  $$

  Now, if $\Gamma'_1 \subset \Exc(\pi_1)$ is any irreducible component with
  strict transform $\wtilde \Gamma'_1 \subset \wtilde Z$, it is clear that the
  $\Top$-form $\pi_1^*(\sigma)$ has (logarithmic) poles along $\Gamma'_1$ if
  and only if $\varphi_1^* \pi_1^*(\sigma) = \psi^*(\sigma)$ has (logarithmic)
  poles along $\wtilde \Gamma'_1$. The proof is then finished once we observe
  that $\Gamma'_1 \subseteq \pi_1^{-1}(\Delta \cup W)$ if and only if $\wtilde
  \Gamma'_1 \subseteq \psi^{-1}(\Delta \cup W)$.
\end{proof}

\section{Finite covering tricks and log canonical singularities}
\label{sec:ext-ext}

\subsection{The finite covering trick}

In order to prove the extension theorem for a given pair $(Z,
\Delta)$, it is often convenient to go to a cover of $Z$ and argue
there. For instance, if $(Z, \Delta)$ is log canonical one might want
to consider local index-one covers where singularities are generally
easier to describe.

\begin{prop}[Finite covering trick]\label{prop:finitecoveringtrick}
  Consider a commutative diagram of surjective morphisms of logarithmic
  pairs as follows,
  $$
  \xymatrix{ (\wtilde X, \wtilde D) \ar[rrr]^{\txt{\scriptsize $\wtilde
        \gamma$, finite}} \ar[d]_{\txt{\scriptsize $\wtilde \pi$\\\scriptsize
        contracts $E_X$}} &&& (\wtilde Z,\wtilde \Delta)
    \ar[d]^{\txt{\scriptsize$\pi$ \\\scriptsize
        log resolution,\\\scriptsize contracts  $E_Z$}} \\
    (X, D) \ar[rrr]_{\txt{\scriptsize $\gamma$, finite}} &&& (Z,\Delta) }
  $$
  where $\wtilde X$ is the normalization of the fiber product $\wtilde
  Z \times_Z {X}$. Let $\Top$ be a reflexive tensor operation, $\sigma
  \in H^0\bigl( Z,\, \Top \Omega^1_Z(\log \Delta) \bigr)$ a
  $\Top$-form, and $E_Z \subset \Exc(\pi) \subset \wtilde Z$ a
  $\pi$-exceptional divisor. Assume that either
  \begin{enumerate-p}
  \item\ilabel{il:extall} $E_Z$ is the union of all $\pi$-exceptional
    components not contained in $\wtilde \Delta$, or

  \item\ilabel{il:extlog} $\Top = \bigwedge^{[p]}$, and no component of $E_Z
    \subset \wtilde Z$ is contained in $\wtilde \Delta$.
  \end{enumerate-p}
  Then
  $$
  \wtilde \pi^* \gamma^{[*]}(\sigma) \in H^0\bigl( \wtilde X,\, \Top
  \Omega^1_{\wtilde X}(\log (\wtilde D + E_X)) \bigr) \, \Longrightarrow \,
  \pi^*(\sigma) \in H^0\bigl( \wtilde Z,\, \Top \Omega^1_{\wtilde Z}(\log
  (\wtilde \Delta+E_Z)) \bigr),
  $$
  where $E_X := \supp(\wtilde \gamma^{-1}(E_Z))$ is the reduced preimage of
  $E_Z$.
\end{prop}

\begin{subexample}
  If $\Top$ is not of the form $\bigwedge^{[p]}$, the assumption made
  in~\iref{il:extall} is indeed necessary.  For an example in the simple case
  where $\Top = \Sym^{[2]}$ and $\Delta = \emptyset$, let $\wtilde Z$ be the
  total space of $\sO_{\bP^1}(-2)$, and $E_Z$ the zero-section. It is
  reasonably easy to write down a form
  $$
  \sigma \in H^0\bigl( \wtilde Z,\, \Sym^2 \Omega^1_{\wtilde Z}(\log
  E_Z)\bigr) \setminus H^0\bigl( \wtilde Z,\, \Sym^2 \Omega^1_{\wtilde Z}
  \bigr).
  $$
  Because $E_Z$ contracts to a quotient singularity that has a smooth 2:1
  cover, this example shows that the conclusion of
  Proposition~\ref{prop:finitecoveringtrick} holds only for differentials with
  logarithmic poles along $E_Z$, and that the boundary given there is indeed
  the smallest possible.

  \PreprintAndPublication{In order to give an explicit example for
    $\sigma$, consider the standard coordinate cover of $\wtilde Z$
    with open sets $U_1, U_2 \simeq \mathbb A^2$, where $U_i$ carries
    coordinates $x_i, y_i$ and coordinate change is given as
    $$
    \phi_{1,2} : (x_1, y_1) \mapsto (x_2, y_2) = ( x_1^{-1},\,
    x_1^2y_1).
    $$
    In these coordinates the bundle map $U_i \to \P^1$ is given as $(x_i, y_i)
    \to x_i$, and the zero-section $E_Z$ is given as $E_Z \cap U_i = \{y_i =
    0\}$. Now take
    $$
    \sigma_2 := y_2^{-1}(dy_2)^2 \in H^0 \bigl(U_2,\, \Sym^2(
    \Omega^1_{\wtilde Z}(\log E_Z))\bigr)
    $$
    and observe that $\phi_{1,2}^*(\sigma_2)$ extends to a form in
    $H^0\bigl(U_1,\, \Sym^2(\Omega^1_{\wtilde Z}(\log E_Z))\bigr)$.}{}
\end{subexample}

\begin{proof}[Proof of Proposition~\ref{prop:finitecoveringtrick}]
  Suppose that we are given a $\Top$-form $\sigma \in H^0\bigl( Z,\,
  \Top \Omega^1_Z(\log \Delta) \bigr)$ such that
  \begin{equation}\label{eq:ass123}
    \wtilde \pi^* \gamma^{[*]}(\sigma) \in H^0\bigl( \wtilde X,\, \Top
    \Omega^1_{\wtilde X}(\log (\wtilde D +E_X)) \bigr).
  \end{equation}
  We need to show that $\sigma$ extends to all of $\wtilde Z$ as a
  $\Top$-form, i.e., that
  \begin{equation}\label{eq:aim12}
    \pi^*(\sigma) \in H^0\bigl( \wtilde Z,\, \Top
    \Omega^1_{\wtilde Z}(\log (\wtilde \Delta+E_Z)) \bigr).
  \end{equation}
  Since~\eqref{eq:aim12} holds outside of $\Exc(\pi)$, and since $\Top
  \Omega^1_{\wtilde Z}\bigl(\log (\wtilde \Delta + E_Z)\bigr)$ is locally
  free, it suffices to show~\eqref{eq:aim12} near general points of components
  of $\Exc(\pi)$. Thus, let $E_Z' \subset \Exc(\pi)$ be an irreducible
  component and $x \in E_Z'$ a general point. Over a suitably small
  neighborhood of $x$, the morphism $\wtilde \gamma$ is branched only along
  $E_Z'$, if at all.

  We will apply Corollary~\ref{cor:pb2} for this small neighborhood of $x$.
  If $E_Z' \subseteq \wtilde \Delta+E_Z$, then \eqref{eq:aim12} follows
  from~\eqref{eq:ass123} by (\ref{cor:pb2}.\ref{il:pbcrit1}). This proves the
  statement in case~\iref{il:extall}. If $E_Z' \not \subseteq \wtilde
  \Delta+E_Z$, we are in case~\iref{il:extlog}, so $\Top =
  \bigwedge^{[p]}$. Then Inclusion~\eqref{eq:aim12} follows
  from~\eqref{eq:ass123} by (\ref{cor:pb2}.\ref{il:pbcrit2}). This proves the
  statement in case~\iref{il:extlog}.
\end{proof}

The following are two immediate consequences of
Proposition~\ref{prop:finitecoveringtrick}.

\begin{cor}\label{cor:c1}
  Let $(Z, \Delta)$ be a logarithmic pair, $\Top$ a reflexive tensor operation
  and assume that there exists a finite morphism of pairs $\gamma: (X, D) \to
  (Z, \Delta)$ such that the extension theorem holds for $\Top$-forms on $(X,
  D)$, in the sense of Definition~\ref{def:extensionthmholds}.  Then the
  extension theorem holds for $\Top$-forms $(Z, \Delta)$.
\end{cor}
 \begin{proof}
   Let $\pi: (\wtilde Z, \wtilde \Delta) \to (Z, \Delta)$ be a
   log resolution and consider the snc divisor
   $$
   \Gamma_Z := \supp \big( \wtilde \Delta \cup \Exc(\pi) \bigr).
   $$
   Further let $U \subseteq Z$ be an open set and $$ \sigma \in H^0\bigl(
   U\setminus{(Z,\Delta)_{\sing}},\, \Top \Omega^1_Z(\log \Delta )\bigr) =H^0
   \bigl( \pi^{-1}(U) \setminus \Exc(\pi),\, \Top \Omega^1_{\wtilde Z}(\log
   \Gamma_Z) \bigr)
   $$
   a $\Top$-form defined away from the singularities. We need to show that its
   pull-back extends to a $\Top$-form on $\bigl( \pi^{-1}(U),\, \Gamma_Z
   \big)$, i.e.,
   \begin{equation}\label{eq:want}
     \pi^*(\sigma) \in H^0\bigl( \pi^{-1}(U),\, \Top
     \Omega^1_{\wtilde Z}(\log \Gamma_Z) \bigr).
   \end{equation}
   For convenience of notation, we shrink $Z$ and assume without loss
   of generality that $U=Z$. In order to prove~\eqref{eq:want},
   consider a commutative diagram of surjective morphisms of pairs,
   \begin{equation}\label{eq:diagcover}
     \xymatrix{
       (\wtilde X, \wtilde D) \ar[rrr]^{\txt{\scriptsize $\wtilde
           \gamma$, finite}} \ar[d]_{\wtilde \pi} &&& (\wtilde Z,\wtilde \Delta)
       \ar[d]^{\txt{\scriptsize$\pi$ \\\scriptsize
           log resolution}} \\
       (X, D) \ar[rrr]_{\txt{\scriptsize $\gamma$, finite}} &&& (Z,\Delta), }
   \end{equation}
   where $\wtilde X$ is the normalization of the fiber product.  Let
   $$
   \Gamma_X := \supp \bigl(\wtilde \gamma^{-1}(\Gamma_Z) \bigr) = \supp
   \bigl(\wtilde D \cup \Exc(\wtilde \pi) \bigr).
   $$
   Then it follows from Lemma~\ref{lem:oneresolutionsuffices} that $\wtilde
   \pi^* \gamma^*(\sigma)$ extends to a $\Top$-form on $\bigl( \wtilde X,\,
   \Gamma_X \bigr)$, i.e.,
   $$
   \wtilde \gamma^* \pi^*(\sigma) = \wtilde \pi^* \gamma^*(\sigma) \in
   H^0\bigl( \wtilde X,\, \Top \Omega^1_{\wtilde X}(\log \Gamma_X) \bigr).
   $$
   Since $\Exc(\pi) \subseteq \Gamma_Z$, \eqref{eq:want} follows from
   case~\iref{il:extall} of Proposition~\ref{prop:finitecoveringtrick} with
   $E_Z := \Exc(\pi) \setminus \wtilde \Delta$.
 \end{proof}

\begin{cor}\label{cor:c2}
  In order to prove the Extension Theorem~\ref{thm:extension-lc} in
  full generality, it suffices to show it under the additional
  assumption that $K_Z + \Delta$ is Cartier.
\end{cor}
\begin{proof}
  Assume that Theorem~\ref{thm:extension-lc} has been shown for all
  log canonical logarithmic pairs whose log canonical divisor is
  Cartier.  Let $(Z, \Delta)$ be an arbitrary logarithmic pair that is
  log canonical with log resolution $\pi: (\wtilde Z, \wtilde \Delta)
  \to (Z, \Delta)$ and assume we are given an open subset $U \subseteq
  Z$ and a form $\sigma \in H^0\bigl( U,\, \Omega^{[p]}_Z(\log \Delta
  )\bigr)$, with $p \in \{\dim Z, \dim Z-1, 1\}$. We need to show that
  \begin{equation}\label{eq:want2}
    \pi^*(\sigma) \in H^0\bigl( \wtilde U,\, \Omega^{[p]}_{\wtilde Z}(\log
    \wtilde \Delta_{\lc}) \bigr),
  \end{equation}
  where $\wtilde U := \pi^{-1}(U)$ and
  $$
  \wtilde \Delta_{\lc} := \text{largest reduced divisor contained in }
  \pi^{-1}\bigl(\Delta \cup \text{non-klt locus of }(Z, \Delta)\bigr).
  $$

  Since the assertion of Theorem~\ref{thm:extension-lc} is local on $Z$ in the
  Zariski topology, we can shrink $Z$ and assume without loss of generality
  that $U = Z$, and that $K_Z+\Delta$ is $\Q$-torsion, i.e., that there exists
  a number $m \in \mathbb N^+$ such that $\sO_Z\bigl(m(K_Z+\Delta)\bigr) \cong
  \sO_Z$. Let $\gamma: (X, D) \to (Z, \Delta)$ be the associated
  index-one-cover, as described in in \cite[2.52]{KM98} or
  \cite[Sect.~3.6f]{Reid87}.  By the inductive assumption, the statement of
  Theorem~\ref{thm:extension-lc} holds for the pair $(X,D)$.

  Since $\gamma$ branches only over the singular points of $(Z,
  \Delta)$, if at all, \cite[5.20]{KM98} immediately gives that $(X,
  D)$ is again log canonical.  Better still, \cite[5.20]{KM98} implies
  that
  $$
  \text{non-klt locus of }(X, D) \subseteq \gamma^{-1} \bigl(\text{non-klt
    locus of }(Z, \Delta)\bigr).
  $$
  Thus, defining $\wtilde X$ as the normalization of $X\times_Z\wtilde
  Z$, $\wtilde\pi:\wtilde X\to X$ the natural morphism, and setting
  $$
  \wtilde D_{\lc} := \text{largest reduced divisor contained in } \wtilde
  \pi^{-1} \bigl( D \cup \text{non-klt locus of }(X,D) \bigr),
  $$
  gives that $\wtilde D_{\lc} \subseteq \gamma^{-1}(\wtilde \Delta_{\lc})$. Now,
  applying the argument from the proof of Corollary~\ref{cor:c1} along with
  case~\iref{il:extlog} of Proposition~\ref{prop:finitecoveringtrick} implies
  \eqref{eq:want2}, as desired.
\end{proof}

\subsection{Finitely dominated and \blc{} pairs}
\label{ssec:dlc}

It follows from Corollary~\ref{cor:c1} that the extension theorem holds for
pairs with quotient singularities, or in fact for pairs that can be locally
finitely dominated by snc pairs. Surface singularities that appear in minimal
model theory often have this property. Because of their importance in the
applications, we discuss one class of examples in more detail here.

\begin{defn}[Finitely dominated pair]
  A logarithmic pair $(Z,\Delta)$ is said to be \emph{finitely
    dominated by analytic snc pairs} if for any point $z \in Z$, there
  exists an analytic neighborhood $U$ of $z$ and a finite, surjective
  morphism of logarithmic pairs $(\wtilde U, D) \to (U,\Delta\cap U)$
  where $\wtilde U$ is smooth and the divisor $D$ has only simple
  normal crossings.
\end{defn}

\begin{rem}\label{rem:extthmforfdsap}
  By Corollary~\ref{cor:c1}, if $\Top$ is any reflexive tensor
  operation, then the extension theorem holds for $\Top$-forms on any
  pair $(Z,\Delta)$ that is finitely dominated by analytic snc pairs.
\end{rem}

\begin{defn}[\blc]\label{def:dlc}
  A logarithmic pair $(Z,\Delta)$ is called \emph{\blc} if
  $(Z,\Delta)$ is log canonical and $(Z\setminus \Delta, \emptyset)$
  is log terminal.
\end{defn}

\begin{example}
  It follows immediately from the definition that dlt pairs are \blc,
  cf.~\cite[2.37]{KM98}. For a less obvious example, let $Z$ be the
  cone over a conic and $\Delta$ the union of two rays through the
  vertex. Then $(Z,\Delta)$ is \blc, but not dlt.
\end{example}

The next example shows how \blc{} pairs appear as limits of dlt pairs.
These limits play an important role in Keel-McKernan's proof of the
Miyanishi conjecture for surfaces, \cite[Sect.~6]{KMcK}, and in the
last two authors' recent attempts to generalize Shafarevich
hyperbolicity to families over higher dimensional base manifolds,
\cite{KK07b, KK08c}, see also~\cite{KS06}.

\begin{example}
  Let $(Z, \Delta)$ be a log canonical logarithmic pair.  Suppose that $\Delta$ is
  $\Q$-Cartier and that for any positive, sufficiently small rational number
  $\varepsilon \in \Q^+$, the non-reduced pair $\bigl( Z, (1-\varepsilon)\Delta
  \bigr)$ is dlt, or equivalently klt. Then $(Z,\Delta)$ is \blc{}.
\end{example}

\begin{lem}\label{ex:dltsurfisfd}
  Let $(Z,\Delta)$ be a \blc{} pair of dimension $2$. Then $Z$ is
  $\mathbb Q$-factorial and $(Z,\Delta)$ is finitely dominated by
  analytic snc pairs. In particular, dlt surface pairs are finitely
  dominated by analytic snc pairs.
\end{lem}

The proof of Lemma~\ref{ex:dltsurfisfd} uses the notion of discrepancy, which
we recall for the reader's convenience.

\begin{defn}[\protect{Discrepancy, cf.~\cite[Sect.~2.3]{KM98}}]\label{defn:discrep}
  Let $(Z, \Delta)$ be a logarithmic pair and let $\pi: (\wtilde Z, \wtilde
  \Delta) \to (Z, \Delta)$ be a log-resolution. If $\wtilde \Delta' \subset
  \wtilde \Delta$ is the strict transform of $\Delta$, the $\mathbb
  Q$-divisors $K_{\wtilde Z} + \wtilde \Delta'$ and $\pi^*\bigl( K_Z + \Delta
  \bigr)$ differ only by a $\mathbb Q$-linear combination of exceptional
  divisors. We can therefore write
  $$
  K_{\wtilde Z} + \wtilde \Delta' = \pi^*\bigl( K_Z + \Delta \bigr) + \sum_{%
    \begin{matrix} 
      \text{\footnotesize $E_i\subset \wtilde Z$}\\ 
      \hskip-10ex
      \text{\footnotesize $\pi$-exceptional 
        divisors}\hskip-10ex
    \end{matrix}
  }
  a(E_i, Z, \Delta) \cdot E_i.
  $$
  The rational number $a(E_i, Z, \Delta)$ is called the \emph{discrepancy of
    the divisor $E_i$}.
\end{defn}

\begin{proof}[Proof of Lemma~\ref{ex:dltsurfisfd}]
  Let $z \in (Z,\Delta)_{\sing}$ be an arbitrary singular point. If $z
  \not \in \Delta$, then the statement follows from \cite[4.18]{KM98}.
  We can thus assume without loss of generality for the remainder of
  the proof that $z \in \Delta$.

  Next observe that for any rational number $0<\varepsilon<1$, the
  non-reduced pair $\left(Z,(1-\varepsilon) \Delta\right)$ is
  \emph{numerically dlt}; see \cite[4.1]{KM98} for the definition and
  use \cite[3.41]{KM98} for an explicit discrepancy computation. By
  \cite[4.11]{KM98}, $Z$ is then $\bQ$-factorial. Using
  $\bQ$-factoriality, we can then choose a sufficiently small Zariski
  neighborhood $U$ of $z$ and consider the index-one cover for
  $\Delta\cap U$. This gives a finite morphism of pairs $\gamma :
  (\wtilde U, \wtilde \Delta) \to (U,\Delta\cap U)$, where the
  morphism $\gamma$ is branched only over the singularities of $U$,
  where $\gamma^{-1}(z) = \{\wtilde z\}$ is a single point, and where
  $\wtilde \Delta= \gamma^*(\Delta\cap U)$ is
  Cartier---see~\cite[5.19]{KM98} for the construction. Since
  discrepancies only increase under taking finite covers,
  \cite[5.20]{KM98}, the pair $(\wtilde U, \wtilde \Delta)$ will again
  be \blc{}. In particular, it suffices to prove the claim for a
  neighborhood of $\wtilde z$ in $(\wtilde U, \wtilde \Delta)$. We can
  thus assume without loss of generality that $z \in \Delta$ and that
  $\Delta$ is Cartier in our original setup.

  Next, we claim that $(Z, \emptyset)$ is canonical at $z$.  In fact,
  let $E$ be any divisor centered above $z$, as in \cite[2.24]{KM98}.
  Since $z \in \Delta$, and since $\Delta$ is Cartier, the pull-back
  of $\Delta$ to any resolution where $E$ appears will contain $E$
  with multiplicity at least $1$. In particular, we have the following
  inequality of discrepancies: $0 \leq a(E, Z, \Delta) +1 \leq a(E, Z,
  \emptyset)$.  This shows that $(Z, \emptyset)$ is canonical at $z$
  as claimed.

  By \cite[4.20-21]{KM98}, it is then clear that $Z$ has a Du~Val
  quotient singularity at $z$. Again replacing $Z$ by a finite cover
  of a suitable neighborhood of $z$, and replacing $z$ by its preimage
  in the covering space, we can henceforth assume without loss of
  generality that $Z$ is smooth. But then the claim follows from
  \cite[4.15]{KM98}.
\end{proof}

\begin{rem}
  It follows from a result of Brieskorn, \cite{Brieskorn68}, that any
  two-dimensional pair $(X, \Delta)$ that is finitely dominated by
  analytic snc pairs has quotient singularities in the following
  sense: For every point $x \in X$ there exists a finite subgroup $G
  \subset GL_2(\C)$ without quasi-reflections, an analytic
  neighborhood $U$ in $X$, and a biholomorphic map $\varphi: U \to V$
  to an analytic neighborhood $V$ of $\pi(0,0)$ in $\C^2 /G$, where
  $\pi: \C^2 \to \C^2/G$ denotes the quotient map. Furthermore, the
  preimage $\pi^{-1}(\varphi(\Delta \cap V))$ coincides with the
  intersection of $a_1D_1 + a_2D_2$ with $\pi^{-1}(V)$, where $a_j \in
  \{0,1\}$ and $D_j = \{z_j = 0\} \subset \C^2$.
\end{rem}

\section{Vector fields and local group actions on singular spaces}

In this section, we discuss vector fields on singular complex spaces
and their relation to local Lie group actions. We will then show that
local group actions and vector fields lift to functorial resolutions.
This will be used in the proof of the extension theorem for $(n-1)$-
forms in Section~\ref{sec:ext-n-1}.

\subsection{Local actions and logarithmic vector fields}

For the reader's convenience, we recall the standard definition of a
local group action.

\begin{defn}[Local group action, \protect{cf.~\cite[Sect.~4]{Kaup65}}]
  Let $G$ be a connected complex Lie group and $Z$ a reduced complex
  space. A \emph{local $G$-action} is given by a holomorphic map
  $\varPhi: \Theta \to Z$, where $\Theta$ is an open neighborhood of
  the neutral section $\{e\} \times Z$ in $G \times Z$ such that
  \begin{enumerate}
  \item for all $z \in Z$ the subset $\Theta(z) := \{g \in G \mid
    (g,z) \in \Theta \}$ is connected,
  \item \ilabel{item:localaction}setting $\varPhi(g,z) =: g\acts z$,
    we have $e \acts z = z$ for all $z \in Z$, and if $(gh,z) \in
    \Theta$, if $(h,z) \in \Theta$ and $(g, h\acts z) \in \Theta$,
    then $(gh)\acts z = g\acts(h\acts z)$ holds.
  \end{enumerate}
\end{defn}

There is a natural notion of equivalence of local $G$-actions on $Z$
given by shrinking $\Theta$ to a smaller neighborhood of $\{e\} \times
Z$ in $G\times Z$. To an equivalence class of actions one assigns a
linear map $\lambda$ from the Lie algebra $\mathfrak{g}$ of $G$ into
the Lie algebra $H^0\big(Z,\, \sT_Z\bigr)$ of vector fields on $Z$, as
follows. If $\xi \in \mathfrak{g}$ is any element of the Lie algebra,
its image $\xi_Z = \lambda(\xi)$ is defined by the equation
$$
\xi_Z(f)(z) = \left.\frac{d}{dt}\right|_{t=0}f\bigl( \exp_G (-t\xi)
\acts z \bigr),
$$
where $f$ is an arbitrary holomorphic function defined near $z$ and
$\exp_G: \mathfrak{g} \to G$ is the exponential map of $G$. If we
consider $\mathfrak{g}$ as the Lie algebra of left-invariant vector
fields on $G$, the map $\lambda$ is a homomorphism of Lie
algebras. The converse statement is a classical result of complex
analysis.

\begin{fact}[\protect{Vector fields and local group actions, \cite[Satz~3]{Kaup65}}]
  If $\lambda: \mathfrak{g} \to H^0\bigl(Z,\, \sT_Z\bigr)$ is a
  homomorphism of Lie algebras, then up to equivalence, there exists a
  unique local $G$-action on $Z$ that induces the given $\lambda$. In
  particular, any vector field $\eta \in H^0\bigl(Z,\, \sT_Z\bigr)$
  induces a local $\mathbb{C}$-action $\Phi_\eta$ on $Z$. \qed
\end{fact}

We also note that if $(Z,\Delta)$ is a logarithmic pair, then the
local $\mathbb{C}$-actions stabilizing $\Delta$ are precisely the ones
that correspond to logarithmic vector fields, i.e.\ global sections of
$\sT_Z (-\log \Delta)$.

\medskip

The next result is crucial for the lifting property of local group
actions.

\begin{lem}[Smoothness of the action map]\label{lem:flatlocalactions}
  The action map $\varPhi: \Theta \to Z$ of a local $G$-action is
  smooth, i.e., a flat submersion.
\end{lem}

\begin{proof}
  Since the map $\varPhi$ is locally equivariant, it suffices to show
  that it is smooth at points of the form $(e, z) \in \Theta$. Given
  such a point $(e, z) \in \Theta$, there exits an open neighborhood
  $\varXi=\varXi(e)$ of the identity $e\in G$ and two open
  neighborhoods $U, U'$ of $z$ in $Z$ such that
  $$
  \varPsi: \varXi \times U \to \varXi \times U', (g, z) \mapsto (g,
  \varPhi(g,z))
  $$
  is well-defined. The map $\varPsi$ is an open embedding; in
  particular, it is smooth. If we denote the canonical (smooth)
  projection by $\pi_2: \varXi \times U' \to U'$, the claim follows
  from the observation that $\varPhi|_{\varXi \times U} = \pi_2 \circ
  \varPsi$ is the composition of smooth morphisms.
\end{proof}

\subsection{Lifting vector fields to functorial resolutions}

Unlike in the surface case, there is no notion of a 'minimal
resolution of singularities' in higher dimensions. There is, however,
a canonical resolution procedure that has certain universal
properties. We briefly recall the relevant facts.

\begin{thm}[\protect{Functorial resolution of singularities, cf.~\cite[Thm.~3.35, 3.45]{Kollar07}}]
  There exists a \emph{resolution functor} $\mathcal{R}: (Z, \Delta)
  \to \bigl(\pi_{Z,\Delta}: R(Z, \Delta) \to (Z, \Delta) \bigr)$ that
  assigns to any logarithmic pair $(Z, \Delta)$ a new pair $R(Z,
  \Delta)$ and a morphism $\pi_{Z,\Delta} : R(Z, \Delta) \to (Z,
  \Delta)$, with the following properties.
  \begin{enumerate}
  \item The morphism $\pi := \pi_{Z, \Delta}: R(Z,\Delta) \to
    (Z,\Delta)$ is a log resolution of $(Z, \Delta)$.
  \item The morphism $\pi$ is projective over any compact subset of
    $Z$.
  \item The functor $\mathcal{R}$ commutes with smooth holomorphic
    maps. That is to say that for any smooth morphism $f: (X,D) \to
    (Z, \Delta)$ of logarithmic pairs there exists a unique smooth
    morphism $\mathcal{R}(f): R(X, D) \to R(Z,\Delta)$ giving a fiber
    product square as follows.
    $$
    \xymatrix{ R(X,D) \ar_{\pi_{X,D}}[d]\ar^{\mathcal{R}(f)}[r]&
      R(Z,\Delta)\ar^{\pi_{Z,\Delta}}[d] \\
      (X,D) \ar^{f}[r]& (Z, \Delta).}
    $$
    \qed
  \end{enumerate}
\end{thm}

\begin{notation}
  We call a log resolution $\pi: (\wtilde Z, \wtilde \Delta) \to (Z,
  \Delta)$ \emph{functorial} if it is of the form $\mathcal{R} (Z,
  \Delta)$.
\end{notation}

\begin{prop}[Lifting of local actions to the functorial resolution]\label{prop:localactionslift}
  Let $\varPhi: \Theta \to Z$ be a local $G$-action on a complex space
  $Z$. Let $\pi: (\wtilde Z, \emptyset) \to (Z, \emptyset)$ be a
  functorial log resolution. Then, $\varPhi$ lifts to a local
  $G$-action on $\wtilde Z$. More precisely, if $\wtilde\Theta :=
  (\Id_G \times \pi)^{-1}(\Theta) \subset G \times \wtilde Z$, then
  there exits a local action $\wtilde \varPhi: \wtilde \Theta \to
  \wtilde Z$ such that the following diagram commutes:
  $$
  \begin{xymatrix}{
      \wtilde\Theta \ar^{\wtilde \varPhi}[r]\ar_{\Id_G \times \pi}[d] &
      \wtilde Z \ar^{\pi}[d]\\
      \Theta \ar^{\varPhi}[r]&Z.  }
  \end{xymatrix}
  $$
  Furthermore, if $(Z,\Delta)$ is a logarithmic pair, if $\pi:
  (\wtilde Z, \wtilde \Delta) \to (Z, \Delta)$ is a functorial log
  resolution, if $\varPhi=\varPhi_\xi$ for some $\xi \in H^0\bigl(Z,\,
  \sT_Z(-\log \Delta)\bigr)$ and if $W$ is any $\varPhi$-invariant
  subvariety of $Z$, we set
  $$
  \wtilde \Delta_W := \text{largest reduced divisor contained in
  }\pi^{-1}(\Delta \cup W).
  $$
  Then, $\wtilde \varPhi$ stabilizes $\wtilde \Delta_W$.
\end{prop}
\begin{proof}
  Using Lemma~\ref{lem:flatlocalactions} and the fact that
  $\mathcal{R}$ commutes with smooth holomorphic maps, we see that the
  application of $\mathcal{R}$ to the diagram
  $$
  G \times Z \hookleftarrow \Theta  \to Z
  $$
  induces a holomorphic map $\wtilde \varPhi: \bigl(\Id_G \times \pi
  \bigr)^{-1}(\Theta)=: \wtilde \Theta \to \wtilde Z $ such that the
  following diagram commutes:
  $$
  \begin{xymatrix}{ G \times \wt Z \ar_{\Id_G \times \pi}[d]&&
      \ar[ll]_(.4){\text{inclusion}} \wtilde\Theta \ar^{\wtilde
        \varPhi}[r]\ar[d]&
      \wtilde Z \ar^{\pi}[d]\\
      G \times Z && \ar[ll]^(.4){\text{inclusion}} \Theta
      \ar^{\varPhi}[r]&Z.  }
  \end{xymatrix}
  $$
  It remains to check that $\wtilde \varPhi: \wtilde \Theta \to
  \wtilde Z$ defines a local $G$-action.  First, notice that $\wtilde
  \Theta$ is an open neighborhood of the neutral section $\{e\} \times
  \wtilde Z$ in $G \times \wtilde Z$. By construction, for a point
  $\wtilde z \in \wtilde Z$ we have
  \begin{equation}\label{eq:fibre}
    \wtilde \Theta (\wtilde z) = \Theta (\pi(\wtilde z)).
  \end{equation}
  Furthermore, we have $g \acts \pi(\wtilde z) = \pi (g \acts \wtilde
  z)$ for all $\wtilde z\in \wtilde Z$ and for all $g \in \wtilde
  \Theta (\wtilde z)$.  It immediately follows that $\wtilde \Theta
  (\wtilde z)$ is connected for all $\wtilde z \in \wtilde Z$. Since
  the biholomorphic map $\wtilde \varPhi_e: \wtilde Z \to \wtilde Z$
  fixes any point in $\wtilde Z \setminus \Exc(\pi)$, it coincides
  with $\Id_{\wtilde Z}$. Given $\wtilde z \in \wtilde Z$ let $g, h
  \in G$ be such that the assumptions of \iref{item:localaction} are
  fulfilled.  By \eqref{eq:fibre} there exists an open neighborhood
  $U$ of $\pi(\wtilde z)$ in $Z$ such that both $\wt \varPhi_{gh}$ and
  $\wt \varPhi_{g} \circ \wt \varPhi_h$ are defined on
  $\pi^{-1}(U)$. Since they coincide on $\pi^{-1}(U) \setminus
  \Exc(\pi)$, they coincide at $\wtilde z$. Hence, we have shown that
  $\wtilde \varPhi: \wtilde \Theta \to \wtilde Z$ is a local
  $G$-action.

  If $(Z, \Delta)$ is a logarithmic pair, if $\xi \in
  H^0\bigl(Z,\,\sT_Z(-\log \Delta)\bigr)$ is a logarithmic vector
  field, and if $W$ is a $\varPhi_\xi$-invariant subvariety of $Z$,
  for all $\wtilde z \in \wt \Delta_W$ and for all $g \in \wtilde
  \Theta (\wtilde z) = \Theta (\pi(\wtilde z))$ we have $\pi(g \acts
  \wtilde z) = g \acts \pi(\wtilde z) \in \Delta \cup W$. Since
  $\wtilde \Theta (\wtilde z)$ is connected, this shows the claim.
\end{proof}

\begin{cor}\label{cor:vectorfieldslift}
  Let $(Z, \Delta)$ be a logarithmic pair, $\pi: (\wtilde Z, \wtilde
  \Delta)\to(Z,\Delta)$ a functorial log resolution and $W$ a
  subvariety of $Z$ that is invariant under any local automorphism of
  $(Z, \Delta)$. Set
  $$
  \wtilde \Delta_W := \text{largest reduced divisor contained in }
  \pi^{-1}(\Delta \cup W).
  $$
  Then $\pi_*\sT_{\wtilde Z}(- \log \wtilde \Delta_W)$ is reflexive.
\end{cor}
\begin{proof}
  Let $U \subset Z$ be an open subset and let $\xi \in H^0 \bigl( U
  \setminus (Z, \Delta)_{\mathrm{sing}},\, \sT_Z (-\log \Delta)\bigr)$
  be a vector field. Since $\sT_Z(-\log \Delta) = \Omega_Z^1(\log
  \Delta)^*$ is reflexive, $\xi$ extends to a logarithmic vector field
  on $U$, i.e., to an element $\xi \in H^0 \bigl( U,\, \sT_Z(-\log
  \Delta)\bigr)$. Lifting the local $\C$-action $\varPhi_\xi$ that
  corresponds to $\xi$ with the help of
  Proposition~\ref{prop:localactionslift}, we obtain a a local
  $\C$-action on $\pi^{-1}(U)$ that stabilizes $\wtilde \Delta_W$. The
  corresponding vector field $\wtilde \xi \in H^0 \bigl(\pi^{-1}(U),\,
  \sT_{\wtilde Z}(-\log \wtilde \Delta_W)\bigr)$ is an extension of
  $\xi$ considered as an element of $H^0 \bigl(\pi^{-1}(U\setminus (Z,
  \Delta)_{\mathrm{sing}}),\, \sT_{\wtilde Z}(-\log \wtilde
  \Delta_W)\bigr)$.
\end{proof}

\part{EXTENSION THEOREMS FOR LOG CANONICAL PAIRS}

\section{Proof of Theorem~\ref*{thm:extension-lc} for $n$-forms}
\label{sec:ext-n}

In this section, we consider the extension problem for logarithmic
$n$-forms. The proof of the case $p=n$ of Theorem~\ref{thm:extension-lc}
immediately follows from the following, slightly stronger result. The
\emph{discrepancy} of an exceptional divisor has been introduced in
Definition~\ref{defn:discrep} above.

\begin{prop}\label{prop:nforms}
  Let $(Z, \Delta)$ be an $n$-dimensional log canonical logarithmic
  pair. Let $\pi: (\wtilde Z, \wtilde \Delta) \to (Z, \Delta)$ be a
  log-resolution and $E_{\lc} \subset \wtilde Z$ the union of all
  $\pi$-exceptional prime divisors $E \not \subseteq \wtilde \Delta$
  with discrepancy $a(E,Z, \Delta) = -1$, endowed with the structure
  of a reduced subscheme of $\wtilde Z$. Then the sheaf $\pi_*
  \Omega^n_{\wtilde Z} \bigl( \log (\wtilde \Delta + E_{\mathrm{lc}})
  \bigr)$ is reflexive.
\end{prop}
\begin{proof}
  After shrinking $Z$ if necessary, it suffices to show that the
  pull-back of any $n$-form $\sigma \in H^0 \bigl(Z, \Omega^{[n]}_{Z}
  (\log \Delta) \bigr)$ extends to an element of $H^0\bigl( \wtilde Z,
  \Omega_{\wtilde Z}^n (\log (\wtilde \Delta + E_{\mathrm{lc}}))
  \bigr)$. Using the argument in the proof of Corollary~\ref{cor:c2}
  and the discrepancy calculation in the proof of \cite[Prop.\
  5.20]{KM98}, we see that it is sufficient to prove the claim under
  the additional assumption that $K_Z + \Delta$ is Cartier.

  First, we renumber the exceptional prime divisors $E_1, \dots, E_m$
  of $\pi$ in such a way that
  \begin{enumerate-p}
  \item $\pi (E_j) \subset \Delta$ iff $j = 1, \dots, k$,
  \item $a(E_j, Z, \Delta) \geq 0$ for $j =  k+1, \dots, l$,
  \item $E_{\mathrm{lc}} = \bigcup_{j = l+1}^m E_j$.
  \end{enumerate-p} 
  Using the assumption that $a(E_j, Z, \Delta) \geq -1$ for all $j$,
  we obtain that
  \begin{equation}\label{eq:deltatilde}
    K_{\wtilde Z} + \pi_*^{-1}(\Delta) - \sum_{j=1}^k a(E_j, Z,
    \Delta) E_j = K_{\wtilde Z} + \wtilde{\Delta} - 
    \sum_{j=1}^k c_j E_j
  \end{equation}
  for some $c_j \geq 0$.  From~\eqref{eq:deltatilde} and the
  definition of discrepancy we conclude that
  \begin{equation}\label{eq:nformsdiscrepancy}
    \pi^* (K_Z + \Delta) = K_{\wtilde Z} + \wtilde \Delta +
    E_{\mathrm{lc}} - \sum_{j=1}^l b_j E_j, 
  \end{equation}
  for some $b_j \geq 0$ ---note that the $b_j$ are integral because
  $K_Z+\Delta$ is Cartier. Equation~\eqref{eq:nformsdiscrepancy} then
  implies that any $n$-form $\sigma \in H^0 \bigl(Z, \Omega^{[n]}_{Z}
  (\log \Delta) \bigr) = H^0 \bigl(Z, \sO_Z(K_Z +\Delta) \bigr) $ extends to
  an element of $H^0\bigl( \wtilde Z, \Omega_{\wtilde Z}^n (\log
  (\wtilde \Delta + E_{\mathrm{lc}}) )\bigr)$.
\end{proof}

\begin{rem}\label{rem:lcnecessary:n}
  It follows from the proof of Proposition~\ref{prop:nforms} that the
  assumption ``log canonical'' is indeed necessary for the case
  $p = n$ of the Extension Theorem~\ref{thm:extension-lc}.
\end{rem}

\section{Proof of Theorem~\ref*{thm:extension-lc} for $(n-1)$-forms}
\label{sec:ext-n-1}

In this section, we consider the case $p=n-1$ of
Theorem~\ref{thm:extension-lc}. We recall the statement.

\begin{prop}\label{prop:n-1forms}
  Let $(Z, \Delta)$ be a log canonical logarithmic pair of dimension
  $n$.  Let $\pi: (\wtilde Z, \wtilde \Delta) \to (Z, \Delta)$ be a
  log resolution and set
  $$
  \wtilde \Delta_{\lc} := \text{largest reduced divisor contained in }
  \pi^{-1}\bigl(\Delta \cup \text{non-klt locus of }(Z, \Delta)\bigr).
  $$
  Then $\pi_* \Omega^{n-1}_{\wtilde Z}
  (\log\wtilde{\Delta}_{\mathrm{lc}})$ is reflexive.
\end{prop}

\begin{proof}
  After shrinking $Z$, it suffices to show that the pull-back
  $\pi^*\sigma$ of any $\sigma \in H^0 \bigl(Z,\,\Omega^{[n-1]}_Z(\log
  \Delta) \bigr)$ extends to an element of $H^0 \bigl(\wtilde
  Z,\,\Omega^{n-1}_{\wtilde Z}(\log \wtilde \Delta_{\mathrm{lc}})
  \bigr)$. By Corollary~\ref{cor:c2}, we may assume that $K_Z +
  \Delta$ is Cartier, and, possibly after a further shrinking of $Z$,
  that $K_Z + \Delta$ is trivial. Finally, due to
  Lemma~\ref{lem:oneresolutionsuffices} we may assume that $\pi:
  (\wtilde Z, \wtilde \Delta) \to (Z, \Delta)$ is a functorial
  log-resolution.

  Since $\Omega^{[1]}_Z(\log \Delta)^* \cong \sT_Z(- \log \Delta)$,
  there exists a unique logarithmic vector field $\eta \in
  H^0\bigl(Z,\, \sT_Z(- \log \Delta)\bigr)$ that corresponds to
  $\sigma$ via the perfect pairing
  $$
  \Omega^{[1]}_Z(\log \Delta) \times \Omega^{[n-1]}_Z(\log \Delta)
  \to \sO _Z(K_Z + \Delta) \cong \sO _Z.
  $$
  Since the non-klt locus is invariant under the local $\C$-action
  $\varPhi_\eta$ of $\eta$, we can lift $\eta$ to a vector field $\wtilde \eta
  \in H^0 \bigl(\wtilde Z,\, \sT_{\wtilde Z}(-\log \wtilde
  \Delta_{\mathrm{lc}})\bigr)$ using Corollary~\ref{cor:vectorfieldslift}.
  The assumption that $(Z,\Delta)$ is log canonical implies, via a discrepancy
  computation similar to \eqref{eq:nformsdiscrepancy} in the proof of
  Proposition~\ref{prop:nforms}, that $\sO _{\wtilde Z}(K_{\wtilde Z} +
  \wtilde \Delta_{\mathrm{lc}}) \cong \sO_{\wtilde Z}(D)$ for some effective
  divisor $D$ on $\wtilde Z$.  Hence, the logarithmic vector field $\wtilde
  \eta$ corresponds to an element $\wtilde \sigma \in H^0 \bigl(\wtilde Z,\,
  \Omega^{n-1}_{\wtilde Z}(\log \wtilde \Delta_{\mathrm{lc}})\otimes
  \sO_{\wtilde Z}(-D)\bigr)$ via the pairing
  $$
  \Omega^1_{\wtilde Z}(\log \wtilde \Delta_{\mathrm{lc}}) \times
  \Omega^{n-1}_{\wtilde Z}(\log \wtilde \Delta_{\mathrm{lc}}) \to \sO
  _{\wtilde Z}(K_{\wtilde Z} + \wtilde \Delta_{\mathrm{lc}}) \cong
  \sO_{\wtilde Z}(D).
  $$
  This yields the desired extension of $\sigma$.
\end{proof}

\begin{rem}
  If $\pi: (\wtilde Z,\wtilde \Delta) \to (Z, \Delta)$ is a log
  resolution of a log canonical surface pair $(Z,\Delta)$, not only
  $\pi_* \Omega_{\wtilde Z}^1(\log \wtilde \Delta_{\mathrm{lc}})$ but
  also $\pi_* \Omega_{\wtilde Z}^1(\log \wtilde \Delta)$ is reflexive,
  i.e., $1$-forms extend over the exceptional set of $\pi$ without
  acquiring further logarithmic poles, see \cite[Lem.~1.3]{Wahl85}.
\end{rem}

We conclude this section with an example showing that the assumption
``log canonical'' in Theorem~\ref{thm:extension-lc} is necessary also
for the cases $p=1$ and $n-1$, cf.\ Remark~\ref{rem:lcnecessary:n}.

\begin{example}
  Let $Z$ be the affine cone over a smooth curve $C$ of degree $4$ in
  $\P^2$.  Let $\wtilde Z$ be the total space of the line bundle
  $\sO_C(-1)$. Then, the contraction of the zero section $E$ of
  $\wtilde Z$ yields a log-resolution $\pi: (\wtilde Z, \emptyset) \to
  (Z, \emptyset)$. An elementary intersection number computation shows
  that the discrepancy of $E$ with respect to $Z$ is equal to $-2$. If
  $Z = \{f=0\}$ for some quartic form $f$ in three variables $z_0,
  z_1, z_2$, the (rational) differential form
  $$
  \tau= \frac{dz_1 \wedge dz_2}{\partial f / \partial z_0}
  $$
  yields a global generator for $\Omega^{[2]}_Z$,
  cf.~\cite[Ex.~1.8]{Reid87}. Let $\bar \tau := \pi^* (\tau) \in
  H^0\bigl(\wtilde Z,\, \Omega^2_{\wtilde Z}(2E) \bigr)$ be the
  associated rational two-form on $\wtilde Z$, and observe that $\bar
  \tau$, seen as a section in $\Omega^2_{\wtilde Z}(2E)$, does not
  vanish along $E$. Finally, let $\xi$ be the vector field induced by
  the canonical $\C^*$-action on $\wtilde Z$. Contracting $\bar \tau$
  by $\xi$ we obtain a regular $1$-form $\sigma = \imath_{\xi}\bar
  \tau$ on $\wtilde Z \setminus E$ that does not extend to an element
  of $H^0 \bigl(\wtilde Z,\, \Omega^1_{\wtilde Z}(\log E) \bigr)$. To
  see this, let $U$ be an open subset of $C$ such that $\sO_C(-1)|_U$
  is trivial, and such that there exists a local coordinate $z$ on
  $U$. If the bundle projection is denoted by $p: \wtilde Z \to C$,
  consider $\wtilde U := p^{-1}(U)\cong U\times \C$. If $w$ is a
  linear fiber coordinate on $\wtilde U$, we have $\wtilde U \cap
  E=\{w=0\}$. In these coordinates, $\bar \tau|_{\wtilde U} =
  \frac{g(z,w)}{w^2}dz \wedge dw $ for some nowhere vanishing $g \in
  \sO_{\wtilde{U}}(\wtilde U)$, and $\xi|_{\wtilde U} = w
  \frac{\partial}{\partial w}$. Hence, in the chosen coordinates we
  have $\sigma|_{\wtilde U} = -\frac{g(z,w)}{w}dz \notin
  H^0\bigl(\wtilde U,\, \Omega_{\wtilde Z}^1(\log E) \bigr)$.
\end{example}

\section{Proof of Theorem~\ref*{thm:extension-lc} for $1$-forms}
\label{sec:ext-1}

The aim of the present section is to prove the Extension
Theorem~\ref{thm:extension-lc} for 1-forms. This is an immediate consequence
of the following stronger proposition.

\begin{prop}\label{prop:extension}
  Let $(Z, \Delta)$ be a reduced log canonical pair. Let $\pi: \wtilde Z \to Z$ be a
  birational morphism such that $\wtilde Z$ is smooth, the $\pi$-exceptional set
  $\Exc(\pi) \subset \wtilde Z$ is of pure codimension one, and
  $\supp(\pi^{-1}(\Delta) \cup \Exc(\pi))$ is a divisor with simple normal crossings.
  Let
  \begin{equation}\label{eq:deflcdiv}
    \wtilde \Delta_{\lc} := \text{largest reduced divisor contained in
    } \pi^{-1}\bigl(\Delta \cup \text{non-klt locus of }(Z,\Delta)\bigr).
  \end{equation}
  Then the sheaf $\pi_* \Omega^1_{\wtilde Z}(\log \wtilde
  \Delta_{\lc})$ is reflexive.
\end{prop}

\begin{subrem}
  Observe that the morphism $\pi$ in Proposition~\ref{prop:extension}
  need not be a log resolution in the sense of
  Definition~\ref{def:everythinglog2b} as we do not assume that $\pi$
  is isomorphic over the set where $(Z, \Delta)$ is snc. The setup of
  Proposition~\ref{prop:extension} has the advantage that it behaves
  well under hyperplane sections. This makes it easier to proceed by
  induction.
\end{subrem}

We will prove Proposition~\ref{prop:extension} in the remainder of the
present chapter. Since the proof is somewhat involved, we chose to
present it as a sequence of clearly marked and relatively independent
steps.

\subsection{Proof of Proposition~\ref*{prop:extension}: Setup of notation}

For notational convenience, we call a birational morphism admissible
if it satisfies the assumptions made in
Proposition~\ref{prop:extension}.

\begin{notation}[Admissible morphism]
  Throughout this section, if $(X, D)$ is a logarithmic pair, we call
  a birational morphism $\eta: \wtilde X \to X$ \emph{admissible} if
  $\wtilde X$ is smooth, the $\eta$-exceptional set $\Exc(\eta)$ is of
  pure codimension one, and
  $$
  \supp\bigl( \eta^{-1}(D) \cup \Exc(\eta) \bigr)
  $$
  has simple normal crossings.
\end{notation}

\begin{notation}
  In the setup of Proposition~\ref{prop:extension}, we denote the
  irreducible components of $\Exc(\pi)$ by $E_i \subset \wtilde Z$.
  Further, let $T \subset X$ denote the set of fundamental points of
  $\pi^{-1}$. For $x \in T$, let $F_x := \pi^{-1}(x)$ be the
  associated fiber and $F_{x,i} := F_x \cap E_i$ the obvious
  decomposition.
\end{notation}

\subsection{Proof of Proposition~\ref*{prop:extension}: Technical preparations}

To prove Proposition~\ref{prop:extension}, we argue using repeated
hyperplane sections of $Z$. We show that the induced resolutions of
general hyperplanes are again admissible.

\begin{lem}\label{lem:mayrestrict}
  In the setup of Proposition~\ref{prop:extension}, assume that $\dim
  Z > 1$ and let $H \subset Z$ be a general hyperplane section.
  \begin{enumerate-p}
  \item\ilabel{il:hyper1} If $\Delta_H := \supp(H \cap \Delta)$, then
    the pair $(H, \Delta_H)$ is again log canonical.
  \item\ilabel{il:hyper2} If $\wtilde H := \pi^{-1}(H)$, then the
    restricted morphism $\pi|_{\wtilde H}: \wtilde H \to H$ is
    admissible.
  \item\ilabel{il:hyper3} If $\wtilde \Delta_{\wtilde H, \lc}$ is the
    largest reduced divisor contained in
    $$
    \pi^{-1}\bigl( \Delta_H \cup \text{non-klt locus of }(H, \Delta_H) \bigr),
    $$
    then $\wtilde \Delta_{\wtilde H, \lc} \subset \wtilde \Delta_{\lc}
    \cap \wtilde H$.
  \end{enumerate-p}
\end{lem}
\begin{subrem}
  The inclusion $\wtilde \Delta_{\wtilde H, \lc} \subset \wtilde
  \Delta_{\lc} \cap \wtilde H$ of~\iref{il:hyper3} might be strict.
\end{subrem}
\begin{proof}
  Seidenberg's theorem asserts that $H$ is normal,
  cf.~\cite[Thm.~1.7.1]{BS95}.  Recall from \cite[Lem.~5.17]{KM98}
  that discrepancies do not decrease when taking general hyperplane
  sections. It follows that the pair $(H,\Delta_H)$ is log canonical
  since $(Z, \Delta)$ is. This shows~\iref{il:hyper1}.
  Assertion~\iref{il:hyper3} follows from \cite[Lem.~5.17(1)]{KM98}.

  Since $\wtilde H$ is general in its linear system, Bertini's theorem
  guarantees that $\wtilde H$ is smooth. Zariski's Main Theorem
  \cite[V~Thm.~5.2]{Ha77} now asserts that a point $z \in \wtilde Z$
  is in $\Exc(\pi)$ if and only if the fiber that contains $z$ is
  positive dimensional; the same holds for $\pi|_{\wtilde H}$.  By
  construction, we then have that
  \begin{align}
    \label{eq:hyper4} \Exc(\pi|_{\wtilde H}) & =
    \Exc(\pi) \cap \wtilde H\\
    \label{eq:hyper5} \supp \left( \pi|_{\wtilde H}^{-1}(\Delta_H) \cup
      \Exc(\pi|_{\wtilde H}) \right) & = \supp\bigl( \pi^{-1}(\Delta)
    \cup \Exc(\pi) \bigr) \cap \wtilde H.
  \end{align}
  The left hand side of~\eqref{eq:hyper4} is thus of pure codimension
  one in $\wtilde H$, and another application of Bertini's theorem
  implies that the left hand side of~\eqref{eq:hyper5} is a divisor in
  $\wtilde H$ with simple normal crossings. The admissibility asserted
  in \iref{il:hyper2} is thus shown.
\end{proof}

The following elementary corollary of Mumford's contractibility
criterion, \cite[p.~6]{MR0153682} helps in the discussion of linear
systems of divisors supported on fibers over isolated points.

\begin{prop}\label{prop:KMMimproved}
  Let $\phi: \wtilde Y \to Y$ be a projective birational morphism between
  quasi-projective, normal varieties of dimension $\dim Y > 1$ and assume that
  $\wtilde Y$ is smooth. Let $y \in Y$ be a point whose pre-image $\phi^{-1}(y)$ has
  codimension one\footnote{We do not assume that $\phi^{-1}(y)$ has \emph{pure}
    codimension one.}  and let $F_0, \ldots, F_k \subset \supp\bigl( \phi^{-1}(y)
  \bigr)$ be the reduced divisorial components. If all the $F_i$ are smooth and if
  $\sum k_i F_i$ is a non-trivial, effective linear combination, then there exists a
  number $j$, $0 \leq j \leq k$ such that $k_j \not = 0$ and such that
  \begin{equation}\label{eq:platon}
    h^0\bigl(F_j,\, \sO_{\wtilde Y}(\textstyle{\sum}
    k_i F_i )|_{F_j}\bigr) = 0.
  \end{equation}
\end{prop}
\begin{proof}
  If $j$ is any number with $k_j = 0$, then the trivial sheaf
  $\sO_{F_j}$ injects into $\sO_{\wtilde Y}(\sum k_iF_i)|_{F_j}$ and
  equation~\eqref{eq:platon} cannot hold. To prove
  Proposition~\ref{prop:KMMimproved}, it therefore suffices to find a
  number $j$ such that \eqref{eq:platon} holds; the assertion $k_j
  \not = 0$ is then automatic.

  In order to do this consider general hyperplanes $\wtilde H_1,
  \ldots, \wtilde H_{\dim Y-2} \subset \wtilde Y$, and let $\wtilde H
  = \wtilde H_1 \cap \cdots \cap \wtilde H_{\dim Y-2}$ be their
  intersection. Then $\wtilde H$ is a smooth surface and the
  intersections $C_i := \wtilde H \cap F_i$ are smooth curves.  The
  Stein-factorization of $\phi|_{\wtilde H}$,
  $$
  \xymatrix{ \wtilde H \ar[r]_{\alpha}
    \ar@/^.4cm/[rr]^{\phi|_{\wtilde H}} & \wtilde H' \ar[r]_{\beta} & Y,
  }
  $$
  gives $\alpha: \wtilde H\to \wtilde H'$, a birational morphism that
  maps to a normal surface and contracts precisely the curves $C_i
  \subset \wtilde H$. Using Mumford's criterion that the intersection
  matrix $\bigl(C_i\cdot C_j\bigr)_{i,j}$ is negative definite, we see
  that there exists a $j\in\{1,\dots,k\}$ such that
  $$
  \deg_{C_j} \sO_{\wtilde Y}(\textstyle{\sum} k_i F_i)|_{C_j} = C_j
  \cdot \left(\sum k_i F_i\bigl|_{\wtilde H} \right) < 0,
  $$
  where the intersection product in the middle term is that of curves
  on the smooth surface $\wtilde H$, cf.~\cite[Lem.~5-1-7]{KMM87}. In
  particular, any section in $\sigma \in H^0\bigl(F_j,\,
  \sO_{F_j}(\textstyle{\sum} k_i F_i|_{F_j})\bigr)$ vanishes on $C_j$
  and on all of its deformations.  Since the $\wtilde H_i$ are
  general, those deformations dominate $F_j$, and the section $\sigma$
  must vanish on all of $F_j$. This shows \eqref{eq:platon} and
  completes the proof.
\end{proof}

\subsection{Proof of Proposition~\ref*{prop:extension}: Extendability over isolated points}
\label{ssec:extoverdiscr}

Before proving Proposition~\ref{prop:extension} in full generality in
Section~\ref{ssec:p71-fullpf} below, we consider the case where
reflexivity of $\pi_* \Omega^1_{\wtilde Z}(\log \wtilde \Delta_{\lc})$
is already known away from a finite set. This result will be used as
the anchor for the inductive argument used in
Section~\ref{ssec:p71-fullpf}. The argument relies on a vanishing
result of Steenbrink, \cite{Steenbrink85}.

\begin{prop}\label{prop:isolatedextension}
  In the setup of Proposition~\ref{prop:extension}, let $\Sigma
  \subset T$ be a finite set of points. Assume that $\pi_*
  \Omega^1_{\wtilde Z}(\log \wtilde \Delta_{\lc})$ is reflexive away
  from $\Sigma$. Then $\pi_* \Omega^1_{\wtilde Z}(\log \wtilde
  \Delta_{\lc})$ is reflexive.
\end{prop}

\begin{proof}
  If $n := \dim Z = 2$, the result is shown in
  Proposition~\ref{prop:n-1forms} above. We will thus assume for the
  remainder of the proof that $n \geq 3$. Since the assertion is local
  on $Z$, we can shrink $Z$ and assume without loss of generality that
  the following holds.
  \begin{enumerate-p}
  \item The set $\Sigma$ contains only a single point, $\Sigma =
    \{z\}$, and
  \item\ilabel{il:deltahitsx} either $\Delta = \emptyset$ or every
    irreducible component of $\Delta$ contains $z$.
  \end{enumerate-p}
  By Lemma~\ref{lem:oneresolutionsuffices}, we are free to blow up
  $\wtilde Z$ further, if necessary. Thus, we can also assume that the
  following holds:
  \begin{enumerate-p}
    \setcounter{enumi}{\value{equation}}
  \item the reduced fiber $F_z := \bigl( \pi^{-1}(z) \bigr)_{\red}$, and
  \item the divisor $\wtilde \Delta_{\lc}' := (\wtilde \Delta_{\lc} + F_z)_{\red}$
    are simple normal crossings divisors on $\wtilde Z$.
  \end{enumerate-p}

  To prove Proposition~\ref{prop:isolatedextension}, after shrinking
  $Z$ more, if necessary, we need to show that the natural restriction
  map
  \begin{equation}\label{eq:extrest}
    H^0\bigl(\wtilde Z,\, \Omega^1_{\wtilde Z}(\log \wtilde
    \Delta_{\lc})\bigr) \to H^0\bigl(\wtilde Z \setminus F_z,\,
    \Omega^1_{\wtilde Z}(\log \wtilde \Delta_{\lc}) \bigr)
  \end{equation}
  is surjective. The proof proceeds in two steps. First, we show
  surjectivity of \eqref{eq:extrest} when we replace $\wtilde
  \Delta_{\lc}$ by the slightly larger divisor $\wtilde
  \Delta_{\lc}'$.  Surjectivity of \eqref{eq:extrest} is then shown in
  a second step.

  \subsubsection*{Step 1: Extension with logarithmic poles along $\wtilde \Delta_{\lc}'$}

  Since $n \geq 3$, a vanishing result of Steenbrink,
  \cite[Thm.~2.b]{Steenbrink85}, asserts that
  \begin{equation} \label{Steenbrinkvanishing}
    R^{n-1}\pi_* \bigl(\sJ_{\wtilde \Delta_{\lc}'} \otimes \Omega^{n-1}_{\wtilde
      Z}(\log \wtilde \Delta_{\lc}')\bigr) = 0.
  \end{equation}
  The Formal Duality Theorem~\ref{thm:formalduality} on page
  \pageref{thm:formalduality} states that for any locally free sheaf
  $\sF$ on $\wtilde Z$ and any number $0 \leq j\leq n$, there exists
  an isomorphism
  $$
  \left((R^j\pi_* \sF)_z \right)^{\widehat\ } \cong
  H^{n-j}_{\pi^{-1}(z)}\bigl(\wtilde Z,\, \sF^*\otimes \omega_{\wtilde
    Z}\bigr)^*,
  $$
  where $\,\widehat{~}\,$ denotes completion with respect to the
  maximal ideal $\mathfrak{m}_z$ of the point $z \in Z$. Setting $\sF
  := \sJ_{\wtilde \Delta_{\lc}'} \otimes \Omega^{n-1}_{\wtilde Z}(\log
  \wtilde \Delta_{\lc}')$ and using that $\sF^* \otimes \omega_{\wtilde
    Z} \cong \Omega^1_{\wtilde Z}(\log \wtilde \Delta_{\lc}')$ we see
  that the vanishing \eqref{Steenbrinkvanishing} implies that the
  following cohomology with supports vanishes,
  $$
  H^1_{F_z}\bigl(\widetilde Z,\, \Omega^1_{\wtilde Z}(\log \wtilde
  \Delta_{\lc}')\bigr) = \{0\}.
  $$
  The standard sequence for cohomology with supports,
  \cite[III~ex.~2.3e]{Ha77},
  $$
  \cdots \!\to\! H^0\bigl(\wtilde Z,\, \Omega^1_{\wtilde Z}(\log
  \wtilde \Delta_{\lc}')\bigr)\! \to\! H^0\bigl(\wtilde Z \setminus
  F_z,\, \Omega^1_{\wtilde Z}(\log \wtilde \Delta_{\lc}')
  \bigr)\!\to\!  H^{1}_{F_z}\bigl(\widetilde Z,\, \Omega^1_{\wtilde
    Z}(\log \wtilde \Delta_{\lc}')\bigr)\!\to\!  \cdots
  $$
  then shows surjectivity of the restriction map~\eqref{eq:extrest}
  for the larger boundary divisor $\wtilde \Delta_{\lc}'$.

  \subsubsection*{Step 2: Extension as a form with logarithmic poles along $\wtilde
    \Delta_{\lc}$}

  To prove surjectivity of \eqref{eq:extrest}, we will show that the
  natural inclusion
  \begin{equation}\label{eq:icnl}
    H^0\bigl(\wtilde Z,\, \Omega^1_{\wtilde Z}(\log \wtilde
    \Delta_{\lc})\bigr) \to H^0\bigl(\wtilde Z,\, \Omega^1_{\wtilde Z}(\log
    \wtilde \Delta_{\lc}')\bigr)
  \end{equation}
  is surjective. The results of Step~1 will then finish the proof of
  Proposition~\ref{prop:isolatedextension}.

  If $z \in \Delta$, or if $z$ is contained in the non-klt locus, then the
  divisors $\wtilde \Delta$ and $\wtilde \Delta'$ agree after some additional
  shrinking of $Z$, and~\eqref{eq:icnl} is the identity map.  So we may assume
  that $z \not \in \Delta$, and that the pair $(Z,\Delta)$ is log terminal
  (i.e., plt) in a neighborhood of $z$.  Assumption~\iref{il:deltahitsx} then
  asserts that $\Delta = \emptyset$. It follows that $\wtilde \Delta_{\lc} =
  \emptyset$, and that $\wtilde \Delta_{\lc}' = F_z$. In this setup, recall
  the well-known result that $Z$ has only rational singularities at $z$,
  cf.~\cite[Thm.~5.22]{KM98}.  For rational singularities, surjectivity
  of~\eqref{eq:icnl} has been shown by Namikawa, \cite[Lem.~2]{Namikawa01}.
\end{proof}

\subsection{Proof of Proposition~\ref*{prop:extension}: End of Proof}
\label{ssec:p71-fullpf}

To finish the proof of Proposition~\ref{prop:extension}, after
possibly shrinking $Z$, let $\sigma \in H^0\bigl(\wtilde Z\setminus
E,\, \Omega^1_{\wtilde Z}(\log \wtilde \Delta_{\lc})\bigr)$ be any
form defined outside the $\pi$-exceptional set $E := \Exc(\pi)$, and
let $\wtilde \sigma \in H^0\bigl(\wtilde Z,\, \Omega^1_{\wtilde
  Z}(\log \wtilde \Delta_{\lc})(*E)\bigr)$ be its extension to
$\wtilde Z$ as a logarithmic form, possibly with poles along $E$.

We need to show that indeed $\wtilde \sigma$ does not have any poles
as a logarithmic form. More precisely, if $E' \subset E$ is any
irreducible component, then we show that $\wtilde \sigma$ does not
have any poles along $E'$, i.e.,
\begin{equation}\label{eq:aim13}
\wtilde \sigma \in H^0\bigl(\wtilde Z,\, \Omega^1_{\wtilde Z}(\log \wtilde
\Delta_{\lc})(*(E-E'))\bigr).
\end{equation}
To prove this, we proceed by induction on pairs $\bigl(\dim Z, \codim
\pi(E') \bigr)$, which we order lexicographically as indicated in
Table~\ref{tab:lexo}.
\begin{table}[t]
  \centering
  \begin{tabular}{lccccccccccc}
    No. & 1 & 2 & 3 & 4 & 5&6&7&8&9&10&$\cdots$ \\
    \hline
    $\dim Z$ & 2 & 3 & 3 & 4&4&4&5&5&5&5&$\cdots$ \\
    $\codim \pi(E')$ & 2&2&3&2&3&4&2&3&4&5&$\cdots$ \\
    \\
  \end{tabular}
  \caption{Lexicographical ordering of dimensions and codimensions}
  \label{tab:lexo}
\end{table}

For convenience of notation, we renumber the irreducible components
$E_i$ of $E$, if necessary and assume that $E' = E_0$, and that there
exists a number $k$ such that
$$
\{E_0, \ldots, E_k \} = \left\{ E_i \subset E \text{ an irreducible
  component} \,|\, \pi(E_i)=\pi(E_0) \right\}
$$
Further, let $k_i \in \mathbb N$ be the pole orders of $\wtilde \sigma$
along the $E_i$, i.e., the minimal numbers such that
$$
\wtilde \sigma \in H^0\bigl(\wtilde Z,\, \Omega^1_{\wtilde Z}(\log
\wtilde \Delta_{\lc})\otimes \sO_{\wtilde Z}(\textstyle{\sum} k_i
E_i)\bigr).
$$
To prove~\eqref{eq:aim13} it is then equivalent show that $k_0 = 0$.

\subsubsection*{Start of induction}

In case $\dim Z = \codim \pi(E_0)=2$, the set $T$ of fundamental
points is necessarily isolated, and
Proposition~\ref{prop:isolatedextension}
applies\footnote{Alternatively, Proposition~\ref{prop:n-1forms} would
  also apply.}.

\subsubsection*{Inductive step}

Our induction hypothesis is that the extension statement as
in~\eqref{eq:aim13} holds for all log canonical pairs $(X, D)$, for
all admissible morphisms $\pi_X : \wtilde X \to X$, all logarithmic
forms on $\wtilde X$ defined outside the $\pi_X$-exceptional set and
all $\pi_X$-exceptional divisors $E'_X \subset \wtilde X$ where either
$$
\dim X < \dim Z \quad {\rm or} \quad \bigl(\dim X = \dim Z
\, {\rm and} \, \codim \pi_X(E'_X) < \codim \pi(E_0)\bigr).
$$

If $\dim Z = \codim \pi(E_0)$, then the induction hypothesis asserts
that the set of points where $\pi_* \Omega^1_{\wtilde Z}(\log \wtilde
\Delta_{\lc})$ is not already known to be reflexive is at most
finite. But then Proposition~\ref{prop:isolatedextension} again
implies that $\pi_* \Omega^1_{\wtilde Z}(\log \wtilde \Delta_{\lc})$
is reflexive everywhere, and the claim holds. We will therefore assume
without loss of generality for the remainder of this proof that $\dim
Z > \codim \pi(E_0)$, or, equivalently, that $\dim \pi(E_0) > 0$.

Now choose general hyperplanes $H_1, \ldots, H_{\dim \pi(E_0)} \subset
Z$, consider their intersection $H := H_1 \cap \cdots \cap H_{\dim
  \pi(E_0)}$ and its preimage $\wtilde H := \pi^{-1}(H)$.  Setting
$\Delta_H := \supp(\Delta \cap H)$ and $\wtilde H := \pi^{-1}(H)$, a
repeated application of Lemma~\ref{lem:mayrestrict} then guarantees
that the pair $(H, \Delta_H)$ is log canonical, and the restricted
morphism $\pi|_{\wtilde H}$ is admissible. If $\wtilde \Delta_{H,\lc}
\subset \wtilde H$ is the divisor discussed in
Lemma~\ref{lem:mayrestrict}, the induction hypothesis applies to forms
on $\wtilde H$ with logarithmic poles along $\wtilde \Delta_{H,\lc}
\subset \wtilde \Delta_{\lc}|_{\wtilde H}$.

The variety $H$ then intersects $\pi(E_0)$ in finitely many points
which are general in $\pi(E_0)$. Let $z \in H \cap \pi(E_0)$ be one of
them, and let $F_z := \pi^{-1}(z)$ be the fiber over $z$.  Shrinking
$Z$, if necessary, we may assume without loss of generality that $z$
is the only point of intersection, $\{z\} = H \cap \pi(E_0)$. The
fiber $F_z \subset \wtilde H$ will generally be reducible, and need
not be of pure dimension. However, if we set
$$
F_{z,i} := F_z \cap E_i
$$
then an elementary computation of dimensions and codimensions shows
that the first $(k+1)$ intersections, $F_{z,0}, \ldots, F_{z,k}
\subset F_z$, are precisely those irreducible components of $F_z$ that
have codimension one in $\wtilde H$. For $0 \leq i \leq k$ we also
obtain that
$$
F_{z,i} := E_i \cap \wtilde H.
$$
In particular, since $\pi|_{\wtilde H}$ is admissible by
Lemma~\ref{lem:mayrestrict} and the $E_i$ are all smooth by
assumption, Bertini's theorem applies to show that the $(F_{z,i})_{0
  \leq i \leq k}$ are smooth as well. Note that all prerequisites of
Proposition~\ref{prop:KMMimproved} are thus satisfied. We will apply
that proposition later near the end of the proof.

Now consider the standard restriction sequence for logarithmic forms,
cf.~\cite[Lem.~2.13 and references there]{KK08},
$$
\xymatrix{ 0 \ar[r] & N^{*\phantom{\wtilde H/\wtilde Z}}_{\wtilde
    H/\wtilde Z} \ar[r] & \Omega^1_{\wtilde Z}(\log \wtilde
  \Delta_{\lc})|_{\wtilde H} \ar^{\varrho}[r] & \Omega^1_{\wtilde
    H}(\log \wtilde \Delta_{\lc}|_{\wtilde H}) \ar[r] & 0,}
$$
its twist with $\sF := \sO_{\wtilde H}(\sum k_i E_i|_{\wtilde H})$ and
its restriction to $F_{z,j}$, for $0 \leq j \leq k$:
$$
\xymatrix{ N^*_{\wtilde H/\wtilde Z} \otimes \sF \ar[r]^(.4){\alpha}
  \ar[d] & \Omega^1_{\wtilde Z}(\log \wtilde \Delta_{\lc})|_{\wtilde
    H} \otimes \sF \ar[r]^{\beta} \ar[d]_{r_{1,j}} & \Omega^1_{\wtilde
    H}(\log \wtilde \Delta_{\lc}|_{\wtilde H}) \otimes \sF
  \ar[d]_{r_{2,j}} \\
  N^*_{\wtilde H/\wtilde Z} \otimes \sF\bigl|_{F_{z,j}}
  \ar[r]_(.47){\alpha_j} & \Omega^1_{\wtilde Z}(\log \wtilde
  \Delta_{\lc}) \otimes \sF\bigl|_{F_{z,j}} \ar[r]_(.48){\beta_j} &
  \Omega^1_{\wtilde H}(\log \wtilde \Delta_{\lc}|_{\wtilde H}) \otimes
  \sF\bigl|_{F_{z,j}}.}
$$
The induction hypothesis now asserts that $\wtilde \sigma|_{\wtilde
  H}$ is a regular logarithmic form on $\wtilde H$.  More precisely,
using the notation $\varrho: \Omega^1_{\wtilde Z}(\log \wtilde
\Delta_{\lc})|_{\wtilde H} \to \Omega^1_{\wtilde H}(\log \wtilde
\Delta_{\lc}|_{\wtilde H})$ from above, we have
\begin{align}
  \label{eq:arpf-van1} \varrho(\wtilde \sigma|_{\wtilde H}) & \in
  H^0\bigl( \wtilde H,\, \Omega^1_{\wtilde H}(\log \wtilde \Delta_{H,
    \lc})\bigr) &&
  \text{by the induction hypothesis}\\
  \label{eq:arpf-van2} & \subseteq H^0\bigl( \wtilde H,\,
  \Omega^1_{\wtilde H}(\log \wtilde \Delta_{\lc}|_{\wtilde H})\bigr)
  && \text{because } \wtilde \Delta_{H,\lc} \subseteq \wtilde
  \Delta_{\lc}|_{\wtilde H}\,
  \text{\ by \eqref{lem:mayrestrict}}\\
  \label{eq:arpf-van3} & \subseteq H^0\bigl( \wtilde H,\,
  \Omega^1_{\wtilde H}(\log \wtilde \Delta_{\lc}|_{\wtilde H}) \otimes
  \sF\bigr). && \text{because } \sO_{\wtilde H} \subseteq \sF
\end{align}
If $j$ is any number with $k_j > 0$, we can say more. The choice of
the $k_j$ guarantees that $\wtilde \sigma|_{\wtilde H}$ is a section
in $\Omega^1_{\wtilde Z}(\log \wtilde \Delta_{\lc})|_{\wtilde H}
\otimes \sF$ that does not vanish along $\wtilde H \cap E_j$.  On the
other hand, \eqref{eq:arpf-van1}--\eqref{eq:arpf-van3} asserts that
$\beta(\wtilde \sigma|_{\wtilde H})$, i.e., $\varrho(\wtilde
\sigma|_{\wtilde H})$ viewed as a section of $\Omega^1_{\wtilde
  H}(\log \wtilde \Delta_{\lc}|_{\wtilde H}) \otimes \sF$, must
necessarily vanish along $\wtilde H \cap E_j$.  In other words, we
obtain that
$$
r_{1,j}(\wtilde \sigma|_{\wtilde H}) \ne 0 \text{\quad and \quad}
(\beta_j \circ r_{1,j})(\wtilde \sigma|_{\wtilde H}) = (r_{2,j}\circ
\beta)(\wtilde \sigma|_{\wtilde H}) = 0.
$$
In other words, $r_{1,j}(\wtilde \sigma|_{\wtilde H})$ is a
non-trivial section in the kernel of $\beta_j$. Consequently,
$h^0\bigl(F_{z,j},\, N^*_{\wtilde H/\wtilde Z}\otimes \sO_{\wtilde
  H}(\textstyle{\sum}k_i\cdot E_i) \bigr) \ne 0$ for all $j$ with $k_j
> 0$. Note, however, that the restriction of the conormal bundle
$N^*_{\wtilde H/\wtilde Z}$ to $F_z$ ---and hence to $F_{z,j}$--- is
trivial because it is a pull-back from $H$, that is, $N^*_{\wtilde
  H/\wtilde Z} = (\pi|_{\wtilde H})^* (N^*_{H/Z})$.

Summing up, we obtain that
\begin{equation}\label{eq:nearlythere}
  h^0\bigl(F_{z,j},\, \sO_{F_{z,j}}(\textstyle{\sum}k_i E_i|_{F_{z,j}})
  \bigr) \ne 0 \text{ for all $j$ with $k_j > 0$}.
\end{equation}
Now, if there \emph{was} a number $0 \leq j \leq k$ with $k_j > 0$,
then Inequality~\eqref{eq:nearlythere} would clearly contradict
Proposition~\ref{prop:KMMimproved}. It follows that all $(k_j)_{0 \leq
  j \leq k}$ must be zero. In particular, $k_0 = 0$ as claimed.  This
completes the proof of Proposition~\ref{prop:extension} and thus the
proof of Theorem~\ref{thm:extension-lc} for one-forms. \qed

\part{BOGOMOLOV-SOMMESE VANISHING ON SINGULAR SPACES}

\section{Pull-back properties for sheaves of differentials and proof of Theorem~\ref*{thm:Bvanishing}}
\label{sec:bogomolov}

In this section we apply the Extension Theorem~\ref{thm:extension-lc}
to sheaves of reflexive differentials on singular pairs, i.e., sheaves
of differentials that are defined away from the singular set. In good
situations, we show that the pull-back of a sheaf of reflexive
differentials to a log resolution can still be interpreted as a sheaf
of differentials, and that the Kodaira-Iitaka dimension of the sheaves
do not change in the process. The Bogomolov-Sommese Vanishing
Theorem~\ref{thm:Bvanishing} follows as an immediate corollary.

\begin{thm}[Extension  for sheaves of differentials]\label{thm:VZsheafextension}
  Let $(Z,\Delta)$ be a logarithmic pair, and $\pi: (\wtilde Z,
  \wtilde \Delta) \to (Z, \Delta)$ a log resolution. Let $\Top$ be
  a reflexive tensor operation and suppose that there exists a
  reflexive sheaf $\sA$ with inclusion $\iota : \sA \to \Top
  \Omega^1_Z(\log \Delta)$.  Further, assume that one of the following
  two additional assumptions holds:
  \begin{enumerate-p}
  \item\ilabel{il:gpg} the pair $(Z,\Delta)$ is finitely dominated by
    analytic snc pairs, or
  \item\ilabel{il:pgp} the pair $(Z,\Delta)$ is log canonical, the
    sheaf $\sA$ is $\Q$-Cartier and $\Top = \bigwedge^{[p]}$, where $p
    \in \{\dim Z, \dim Z-1, 1\}$.
  \end{enumerate-p}
  Then there exists a factorization
  $$
  \pi^{[*]}\sA \into \sC \into \Top \Omega^1_{\wtilde Z} \bigl(\log
  (\wtilde \Delta+E_\Delta) \bigr),
  $$
  where $E_\Delta \subset \wtilde Z$ is the union of those
  $\pi$-exceptional divisors that are not contained in $\wtilde
  \Delta$, $\sC$ is invertible and $\kappa(\sC) = \kappa(\sA)$.
\end{thm}

\begin{warning}
  Since $ \pi^{[*]}\sA$ is a subsheaf of $\sC$, it might be tempting
  to believe that the equality $\kappa(\sC) = \kappa(\sA)$ is
  immediate. Note, however, that the reflexive tensor products used in
  Definition~\ref{def:kappa1} of the Kodaira-Iitaka dimension
  generally \emph{do not} commute with pull-back. The Kodaira-Iitaka
  dimension $\kappa(\pi^{[*]}\sA)$ could therefore be strictly smaller
  than $\kappa(\sA)$.
\end{warning}

Before proving Theorem~\ref{thm:VZsheafextension} in
Section~\ref{sec:pfVZsheafext} below, we remark that the following,
slightly stronger variant of the Bogomolov-Sommese vanishing
Theorem~\ref{thm:Bvanishing} for log canonical threefolds and surfaces
follows as an immediate corollary to
Theorem~\ref{thm:VZsheafextension}.

\begin{thm}[Bogomolov-Sommese vanishing for log canonical pairs]\label{thm:Bvanishing2}
  Let $(Z, \Delta)$ be a log canonical logarithmic pair. If $p \in
  \{\dim Z, \dim Z-1, 1 \}$ and if $\sA \subset \Omega^{[p]}_Z(\log
  \Delta)$ is any $\Q$-Cartier reflexive subsheaf of rank one, then
  $\kappa(\sA) \leq p$.
\end{thm}
\begin{proof}[Proof of Theorems~\ref{thm:Bvanishing2} and \ref{thm:Bvanishing}]
  We argue by contradiction and assume that there exists a number $p
  \in \{\dim Z, \dim Z-1, 1\}$ and a $\Q$-Cartier reflexive subsheaf
  $\sA \subset \Omega^{[p]}_Z(\log \Delta)$ of rank one, with
  Kodaira-Iitaka dimension $\kappa(\sA) > p$.

  Let $\pi: (\wtilde Z, \wtilde \Delta) \to (Z, \Delta)$ be any log
  resolution.  Theorem~\ref{thm:VZsheafextension} then asserts the
  existence of an invertible sheaf $\sC \subset \Omega^{p}_{\wtilde
    Z}(\log \wtilde \Delta + E_{\Delta})$ with $\kappa(\sC) =
  \kappa(\sA)$.  This contradicts the classical Bogomolov-Sommese
  vanishing theorem for snc pairs, \cite[Cor.~6.9]{EV92}.
\end{proof}

\subsection{Preparations for the proof of Theorem~\ref*{thm:VZsheafextension}}

As a preparation for the proof of Theorem~\ref{thm:VZsheafextension}
we show that the pull-back of a sheaf of reflexive differentials can
be interpreted as a sheaf of differentials if the extension theorem
holds.

\begin{prop}\label{prop:sheafextension1}
  Let $(Z, \Delta)$ be a logarithmic pair, $\Top$ a reflexive
  tensor operation and assume that the extension theorem holds for
  $\Top$-forms on $(Z, \Delta)$, in the sense of
  Definition~\ref{def:extensionthmholds}. If $\pi: (\wtilde Z, \wtilde
  \Delta) \to (Z, \Delta)$ is any log resolution and $E_\Delta \subset
  {\wtilde Z}$ the union of those $\pi$-exceptional components that are
  not contained in $\wtilde \Delta$, then there exists an embedding
  \begin{equation}\label{eq:embed}
  \pi^{[*]} \Top \Omega^1_Z(\log \Delta) \into \Top
  \Omega^1_{\wtilde Z}(\log (\wtilde \Delta + E_\Delta)).
  \end{equation}
\end{prop}
\begin{proof}
  As $\pi$ induces an isomorphism $\wtilde Z \setminus \Exc(\pi)
  \simeq Z \setminus \pi \bigl( \Exc(\pi) \bigr)$, the assumption that
  the extension theorem holds for $\Top$-forms on $(Z, \Delta)$
  immediately implies that
  $$
  \Top \Omega^1_Z(\log \Delta) \simeq \pi_*\Top \Omega^1_{\wtilde Z}(\log
  (\wtilde \Delta + E_\Delta)),
  $$
  because both sides are reflexive and agree in codimension one and
  $Z$ is $S_2$ since it is normal.  Consequently, we obtain a morphism
  $$
  \pi^*\Top \Omega^1_Z(\log \Delta)\simeq \pi^*\pi_* \Top
  \Omega^1_{\wtilde Z}(\log (\wtilde \Delta + E_\Delta)) \to \Top
  \Omega^1_{\wtilde Z}(\log (\wtilde \Delta + E_\Delta)),
  $$
  which is an isomorphism, in particular an embedding, on ${\wtilde
    Z}\setminus \Exc(\pi)$. This remains true after taking the double
  dual of these sheaves.  Therefore the kernel of the map $\pi^{[*]}
  \Top \Omega^1_Z(\log \Delta) \to \Top \Omega^1_{\wtilde Z}(\log
  (\wtilde \Delta + E_\Delta))$ is a torsion sheaf. Since $\pi^{[*]}
  \Top \Omega^1_Z(\log \Delta)$ is torsion-free, this implies the
  statement.
\end{proof}

It is well understood that tensor operations commute with pull-back.
However, this is generally not true for reflexive tensor operations
cf.~\cite{Hassett-Kovacs04}.  Thus, if we are in the setup of
Proposition~\ref{prop:sheafextension1} and if $\sA \subset \Top
\Omega^1_Z(\log \Delta)$ is any sheaf, it is generally not at all
clear if the embedding~\eqref{eq:embed} induces a map between
reflexive tensor products,
$$
\xymatrix{ \pi^{[*]} \sA^{[m]} \ar[r]^(.3){\exists?} & \Sym^m \Top
  \Omega^1_{\wtilde Z}(\log (\wtilde \Delta + E_\Delta)).  }
$$
If the sheaf $\sA$ is invertible, we can obviously say more.

\begin{lem}\label{lem:l1}
  In the setup of Proposition~\ref{prop:sheafextension1}, let $\sA
  \subset \Top \Omega^1_Z(\log \Delta)$ be an invertible subsheaf. If
  $m \in \mathbb N$ is arbitrary, then the embedding~\eqref{eq:embed}
  induces a map
  \begin{equation}\label{eq:emb}
    \pi^{[*]} \sA^{[m]} \into \Sym^m \Top \Omega^1_{\wtilde Z}(\log (\wtilde \Delta +
    E_\Delta)).
  \end{equation}
\end{lem}
\begin{proof}
  Since $\sA$ is invertible, all tensor operations on $\sA$ are
  automatically reflexive. In particular, we have that $\sA^{[m]} =
  \sA^{\otimes m}$ and $\pi^{[*]} \sA^{[m]} \cong \pi^*(\sA^{\otimes
    m}) \cong (\pi^*\sA)^{\otimes m}$.  The existence
  of~\eqref{eq:emb} then follows from
  Proposition~\ref{prop:sheafextension1}. 
\end{proof}

\subsection{Proof of Theorem~\ref*{thm:VZsheafextension}}
\label{sec:pfVZsheafext}

We maintain the notation and the assumptions of
Theorem~\ref{thm:VZsheafextension}. By Theorem~\ref{thm:extension-lc}
or Remark~\ref{rem:extthmforfdsap}, respectively, the extension
theorem holds for the pair $(Z, \Delta)$.
Proposition~\ref{prop:sheafextension1} then gives an embedding
$\psi^{[*]}\sA \into \Sym^{n} \Omega^1_{\wtilde Z}(\log (\wtilde
\Delta + E_\Delta))$.  Let $\sC \subset \Top \Omega^1_{\wtilde Z}
\bigl( \log (\wtilde \Delta + E_\Delta) \bigr)$ be the saturation of
the image, which is automatically reflexive by \cite[Lem.~1.1.16 on
p.~158]{OSS}. By \cite[Lem.~1.1.15 on p.~154]{OSS}, $\sC$ is then
invertible as desired. Further observe that for any $m\in\bN$, the
subsheaf $\sC^{\otimes m} \subset \Sym^m \Top \Omega^1_{\wtilde
  Z}(\log (\wtilde \Delta + E_\Delta))$ is likewise saturated.

\subsubsection{Proof of Theorem~\ref*{thm:VZsheafextension} if $(Z, \Delta)$ is finitely dominated by analytic snc pairs}

If Assumption~\iref{il:gpg} of Theorem~\ref{thm:VZsheafextension}
holds and $m \in \mathbb N$ is arbitrary, then again by
Remark~\ref{rem:extthmforfdsap} and
Proposition~\ref{prop:sheafextension1}, there exists an embedding
$$
\bar \iota^{[m]} : \psi^{[*]}\sA^{[m]} \into \Sym^m \Top
\Omega^1_{\wtilde Z} \bigl(\log (\wtilde \Delta + E_\Delta) \bigr).
$$
It is easy to see that $\bar\iota^{[m]}$ factors through
$\sC^{\otimes m}$ as it does so on the open set where $\psi$ is
isomorphic, and because $\sC^{\otimes m}$ is saturated in the locally
free sheaf $\Sym^m \Top \Omega^1_{\wtilde Z} \bigl(\log (\wtilde \Delta
+ E_\Delta) \bigr)$. It follows that $\kappa(\sC) = \kappa(\sA)$. This
completes the proof in the case when Assumption~\iref{il:gpg} holds. \qed

\subsubsection{Proof of Theorem~\ref*{thm:VZsheafextension} if $(Z, \Delta)$ is log canonical}

It remains to consider the case when Assumption~\iref{il:pgp} of
Theorem~\ref{thm:VZsheafextension} holds. Let $m \in \mathbb N$ and
$\sigma \in H^0\bigl(Z,\, \sA^{[m]} \bigr)$ a section.  Then
$\pi^*(\sigma)$ can be seen as a section in $\sC^{\otimes m}$, with
poles along the exceptional set $E := \Exc(\pi)$, i.e.~$\pi^*(\sigma)
\in H^0\bigl( \wtilde Z, \sC^{\otimes m}(*E) \bigr)$.  To show that
$\kappa(\sC) = \kappa(\sA)$, it suffices to prove that $\pi^*(\sigma)$
does not have any poles as a section in $\sC^{\otimes m}$, i.e., that
\begin{equation}\label{eq:nopole}
  \pi^*(\sigma) \in H^0\bigl(\wtilde Z, \sC^{\otimes m} \bigr)
  \subset H^0\bigl(\wtilde Z, \sC^{\otimes m}(*E) \bigr).
\end{equation}
Since $\sC^{\otimes m}$ is saturated in $\Sym^m \Omega^p_{\wtilde
  Z}(\log (\wtilde \Delta + E_\Delta))$, to show~\eqref{eq:nopole}, it
suffices in turn to show that $\pi^*(\sigma)$ does not have any poles
as a section in the sheaf of symmetric differentials, i.e., that
\begin{equation}\label{eq:nopole2}
  \pi^*(\sigma) \in H^0\bigl( \wtilde Z,\, \Sym^m \Omega^p_{\wtilde Z} (\log
  (\wtilde \Delta + E_\Delta)) \bigr).
\end{equation}
Since that question is local in $Z$ in the analytic topology, we can
shrink $Z$, use that $\sA$ is $\Q$-Cartier and assume without loss of
generality that there exists a number $r$ such that $\sA^{[r]} \cong
\sO_Z$. Similar to the construction in the proof of the finite
covering trick, Proposition~\ref{prop:finitecoveringtrick}, we obtain
a commutative diagram
$$
\xymatrix{ (\wtilde X, \wtilde D) \ar[rrr]^{\txt{\scriptsize
      $\wtilde \gamma$, finite}} \ar[d]_{\txt{\scriptsize $\wtilde
      \pi$\\\scriptsize contracts $\wtilde E$}} &&& (\wtilde Z,
  \wtilde \Delta) \ar[d]^{\txt{\scriptsize$\pi$ \\\scriptsize
      log resolution,\\\scriptsize contracts  $E$}} \\
  (X, D) \ar[rrr]_{\txt{\scriptsize $\gamma$, finite}} &&& (Z,\Delta)
}
$$
where $\gamma$ is the index-one-cover associated to $\sA$, $\wtilde X$
is the normalization of the fiber product $X \times_Z {\wtilde Z}$ and
$\wtilde D \subset \wtilde X$ is the reduced preimage of $\wtilde
\Delta$. As before, let
$$
\wtilde E := \Exc(\wtilde \pi) = \supp \bigl( \wtilde\gamma^{-1}(E)
\bigr) = \supp \bigl( \bigl( \gamma \circ \wtilde \pi \bigr)^{-1}
(Z,\Delta)_{\sing} \bigr)
$$
be the exceptional set of the morphism $\wtilde \pi$. Since $\gamma$
is étale away from the singularities of $Z$, the morphism $\wtilde
\gamma$ is étale outside of $E \subset \wtilde \Delta \cup
E_\Delta$. In particular, the pull-back morphism of differentials
gives an isomorphism
$$
\wtilde \gamma^{[*]} \left( \Sym^m \Omega^{p}_{\wtilde Z}(\log
  \wtilde \Delta + E_\Delta) \right) \cong \Sym^{[m]}
\Omega^{[p]}_{\wtilde X} \bigl( \log \wtilde D+ \wtilde E_D \bigr),
$$
where again $\wtilde E_D \subset \wtilde X$ is union of the $\wtilde
\pi$-exceptional divisors not already contained in $\wtilde D$.  In
order to prove~\eqref{eq:nopole2}, it then suffices to show that
\begin{equation}\label{eq:nopole3}
  \wtilde \gamma^{[*]} \bigl( \pi^*(\sigma) \bigr) = \wtilde \pi^{[*]}
  \gamma^{[*]}(\sigma) \in H^0\bigl(\wtilde X,\, \Sym^{[m]}
  \Omega^{[p]}_{\wtilde X}(\log (\wtilde D + \wtilde E_D)) \bigr),
\end{equation}
cf.~case~(\ref{cor:pb2}.\ref{il:pbcrit1}) of
Corollary~\ref{cor:pb2}. Since the pair $(X, D)$ is again log
canonical by \cite[5.20]{KM98}, Theorem~\ref{thm:extension-lc} applies
to show that the extension theorem holds for $(X, D)$. In particular,
Lemma~\ref{lem:l1} applies to the invertible sheaf $\wtilde \sA :=
\gamma^{[*]}(\sA) \subset \Omega^{[p]}_X(\log
D)$. Inclusion~\eqref{eq:nopole3} follows if one applies the embedding
$$
\wtilde \pi^{[*]} \left( \wtilde \sA^{[m]} \right) \into \Sym^m \Top
\Omega^1_{\wtilde X} \bigl( \log (\wtilde D + \wtilde E_D) \bigr)
$$
to the section $\wtilde \sigma := \gamma^{[*]}(\sigma) \in H^0\bigl(
X,\, \wtilde \sA \bigr)$. This completes the proof of
Theorem~\ref{thm:VZsheafextension} in the case when
Assumption~\iref{il:pgp} holds. \qed

\appendix
\part{APPENDIX}

\section{Duality for cohomology with support}\label{Appendix}

The proof of Proposition~\ref{prop:isolatedextension} relies on the
following version of Hartshorne's Formal Duality Theorem. Since this
is not exactly the version contained in the main reference
\cite{Hartshorne1970}, we recall the relevant facts and include a full
proof for the reader's convenience.

\begin{thm}[\protect{Formal Duality, \cite[Thm.~3.3]{Hartshorne1970}}]\label{thm:formalduality}
  Let $\pi: \wtilde Z \to Z$ be a projective birational morphism of
  quasi-projective varieties, where $\wtilde Z$ is non-singular and
  $Z$ is normal. Let $z \in Z$, and $F := \pi^{-1}(z)$ the fiber over
  $z$. Then, for any locally free sheaf $\sF$ on $\wtilde Z$ and any
  number $0\leq j\leq n$, the exists a canonical isomorphism
  $$
  \left(R^j\pi_* \sF_z\right)^{\widehat\ } \cong
  H^{n-j}_{F}\bigl(\wtilde Z,\, \sF^*\otimes \omega_{\wtilde Z}\bigr)^*,
  $$
  where $\, \widehat\ \,$ denotes completion with respect to the
  maximal ideal $\mathfrak{m}_z$ of the point $z \in Z$.
\end{thm}

We recall a few facts before giving the proof.

\begin{fact}[\protect{Excision for local cohomology, \cite[III~Ex.~2.3f]{Ha77}}]\label{fact:excision}
  Let $Z$ be an algebraic variety, $Y$ a subvariety and $U \subseteq
  Z$ an open subset that contains $Y$. If $i$ is any number and $\sF$
  any sheaf, then there exists a canonical isomorphism $H^i_Y(Z,\,
  \sF) \cong H^i_Y(U,\, \sF\resto{U})$. \qed
\end{fact}

\begin{fact}[\protect{Serre duality on $\wtilde Z$, \cite[III~Thm.7.6]{Ha77}}]\label{fact:serreDuality}
  Let $\wtilde Z$ be a non-singular projective variety of dimension
  $n$. Then there exists a canonical isomorphism
  $$
  H^j(\wtilde Z, \sG) \cong \left(\Ext_{\wtilde Z}^{n-j}(\sG, \omega_{\wtilde
      Z}) \right)^*
  $$
  for all $j\geq 0$, and for every coherent sheaf $\sG$ on $\wtilde
  Z$. \qed
\end{fact}

\begin{fact}[\protect{Approximation of cohomology with support, \cite[Thm~2.8]{Hartshorne1967}}]\label{fact:approx}
  In the notation of Theorem~\ref{thm:formalduality} above, if $\sI$
  is any sheaf of ideals defining the subset $F \subseteq \wtilde Z$,
  the local cohomology groups with support on $F$ and values in a
  coherent algebraic sheaf $\sG$ can be computed as follows:

  \hfill $H^j_F\bigl(\wtilde Z, \sG \bigr) = \underset{\underset{m}{\to}}{\lim}
  \;\Ext_{\wtilde Z}^j \bigl(\factor \sO_{\wtilde Z}.\sI^m., \sG \bigr).$\hfill \qed
\end{fact}

\begin{fact}[\protect{Theorem on Formal Functions, \cite[Ch.~III.11]{Ha77}}]\label{fact:formalFunctions}
  In the notation of Theorem~\ref{thm:formalduality} above, if $\sJ$
  is the $\sO_{\wtilde Z}$-ideal generated by the image of the maximal
  ideal $\mathfrak{m}_z$ under the natural map $\pi^{-1} \sO_Z \to
  \sO_{\wtilde Z}$, and if $\sG$ is any coherent sheaf on $\wtilde Z$,
  then we have
  $$
  \left(R^j\pi_* \sG_z\right)^{\widehat\ } \cong
  \underset{\underset{m}{\leftarrow}}{\lim}\; H^j\bigl(F_m, \sG_m \bigr),
  $$
  where $F_m = \bigl(F, \factor \sO_{\wtilde Z}.\sJ^m.\bigr)$ is
  the $m$-th infinitesimal neighborhood of the fiber $F$, and where
  $\sG_m = \sG \otimes \factor \sO_{\wtilde Z}.\sJ^m.$. \qed
\end{fact}

\begin{fact}[\protect{\cite[Ch.~III.6, Prop~6.7]{Ha77}}]\label{fact:ext-and-tensor}
  Let $\wtilde Z$ be an algebraic variety. For coherent sheaves $\sM$
  and $\sN$ on $\wtilde Z$, we have
  $$
  \Ext_{\wtilde Z}^j(\sF \otimes \sM, \sN) \cong \Ext_{\wtilde
    Z}^j(\sM, \sF^* \otimes \sN)
  $$
  for every locally free sheaf $\sF$ on $\wtilde Z$. \qed
\end{fact}

\begin{proof}[Proof of Theorem~\ref{thm:formalduality}]
  Using the excision theorem for local cohomology,
  Fact~\ref{fact:excision}, we may compactify $Z$ and $\wtilde Z$ and
  assume without loss of generality that both $Z$ and $\wtilde Z$ are
  projective. By Fact~\ref{fact:formalFunctions}, we have
  \begin{equation}\label{eq:rx1}
    \left(R^j\pi_*\sF_z\right)^{\widehat \ } =
    \underset{\leftarrow}{\lim}\; H^j\bigl(F_m, \sF_m \bigr).
  \end{equation}
  The cohomology group on the right hand side of~\eqref{eq:rx1} is
  computed as follows.
  \begin{align*}
    H^j\bigl(F_m, \sF_m \bigr) &= H^j\bigl(\wtilde Z, \sF_m \bigr) \\
    &\cong \left(\Ext_{\wtilde Z}^{n-j}\bigl(\sF_m, \omega_{\wtilde Z}
      \bigr)\right)^* && \text{by Fact~\ref{fact:serreDuality}}\\
    &\cong \left(\Ext_{\wtilde Z}^{n-j} \bigl(\factor \sO_{\wtilde Z}.\sJ^m.,
      \sF^*\otimes \omega_{\wtilde Z} \bigr) \right)^*&& \text{by
      Fact~\ref{fact:ext-and-tensor}}
  \end{align*}
  Substituting this into~\eqref{eq:rx1}, we obtain
  \begin{align*}
    \left(R^j\pi_*\sF_z\right)^{\widehat \ } &\cong
    \underset{\leftarrow}{\lim}\; \left(\Ext_{\wtilde Z}^{n-j}
      \bigl(\sO_{\wtilde Z}/\sJ^m,
      \sF^*\otimes \omega_{\wtilde Z}\bigr) \right)^* \\
    &= \left(\underset{\to}{\lim}\; \Ext_{\wtilde Z}^{n-j} \bigl(\sO_{\wtilde
        Z}/\sJ^m, \sF^*\otimes \omega_{\wtilde Z}
      \bigr)\right)^*\\
    &=\left(H_F^{n-j}\bigl(\wtilde Z, \sF^* \otimes \omega_{\wtilde Z} \bigr)
    \right)^* &&\text{by Fact~\ref{fact:approx}},
  \end{align*}
  as claimed.
\end{proof}


\begin{thebibliography}{KMM87}

\bibitem[BS95]{BS95}
{\sc M.~C. Beltrametti and A.~J. Sommese}: \emph{The adjunction theory of
  complex projective varieties}, de Gruyter Expositions in Mathematics,
  vol.~16, Walter de Gruyter \& Co., Berlin, 1995. {\sf\scriptsize 96f:14004}

\bibitem[Bri68]{Brieskorn68}
{\sc E.~Brieskorn}: \emph{Rationale {S}ingularit\"aten komplexer {F}l\"achen},
  Invent. Math. \textbf{4} (1967/1968), 336--358. {\sf\scriptsize MR0222084 (36
  \#5136)}

\bibitem[Del70]{Deligne70}
{\sc P.~Deligne}: \emph{\'{E}quations diff\'erentielles \`a points singuliers
  r\'eguliers}, Springer-Verlag, Berlin, 1970, Lecture Notes in Mathematics,
  Vol. 163. {\sf\scriptsize 54 \#5232}

\bibitem[EV92]{EV92}
{\sc H.~Esnault and E.~Viehweg}: \emph{Lectures on vanishing theorems}, DMV
  Seminar, vol.~20, Birkh\"auser Verlag, Basel, 1992. {\sf\scriptsize MR1193913
  (94a:14017)}

\bibitem[Fle88]{Flenner88}
{\sc H.~Flenner}: \emph{Extendability of differential forms on nonisolated
  singularities}, Invent. Math. \textbf{94} (1988), no.~2, 317--326.
  {\sf\scriptsize MR958835 (89j:14001)}

\bibitem[Har67]{Hartshorne1967}
{\sc R.~Hartshorne}: \emph{Local cohomology}, A seminar given by A.
  Grothendieck, Harvard University, Fall, vol. 1961, Springer-Verlag, Berlin,
  1967. {\sf\scriptsize MR0224620 (37 \#219)}

\bibitem[Har70]{Hartshorne1970}
{\sc R.~Hartshorne}: \emph{Ample subvarieties of algebraic varieties}, Notes
  written in collaboration with C. Musili. Lecture Notes in Mathematics, Vol.
  156, Springer-Verlag, Berlin, 1970. {\sf\scriptsize MR0282977 (44 \#211)}

\bibitem[Har77]{Ha77}
{\sc R.~Hartshorne}: \emph{Algebraic geometry}, Springer-Verlag, New York,
  1977, Graduate Texts in Mathematics, No. 52. {\sf\scriptsize 57 \#3116}

\bibitem[HK04]{Hassett-Kovacs04}
{\sc B.~Hassett and S.~J. Kov{\'a}cs}: \emph{Reflexive pull-backs and base
  extension}, J. Algebraic Geom. \textbf{13} (2004), no.~2, 233--247.
  {\sf\scriptsize MR2047697 (2005b:14028)}

\bibitem[Iit82]{Iitaka82}
{\sc S.~Iitaka}: \emph{Algebraic geometry}, Graduate Texts in Mathematics,
  vol.~76, Springer-Verlag, New York, 1982, An introduction to birational
  geometry of algebraic varieties, North-Holland Mathematical Library, 24.
  {\sf\scriptsize MR637060 (84j:14001)}

\bibitem[Kau65]{Kaup65}
{\sc W.~Kaup}: \emph{Infinitesimale {T}ransformationsgruppen komplexer
  {R}\"aume}, Math. Ann. \textbf{160} (1965), 72--92. {\sf\scriptsize MR0181761
  (31 \#5988)}

\bibitem[KMM87]{KMM87}
{\sc Y.~Kawamata, K.~Matsuda, and K.~Matsuki}: \emph{Introduction to the
  minimal model problem}, Algebraic geometry, Sendai, 1985, Adv. Stud. Pure
  Math., vol.~10, North-Holland, Amsterdam, 1987, pp.~283--360. {\sf\scriptsize
  MR946243 (89e:14015)}

\bibitem[KK07]{KK07b}
{\sc S.~Kebekus and S.~J. Kov{\'a}cs}: \emph{The structure of surfaces mapping
  to the moduli stack of canonically polarized varieties}, preprint, July 2007.
  {\sf\scriptsize arXiv:0707.2054}

\bibitem[KK08a]{KK08}
{\sc S.~Kebekus and S.~J. Kov{\'a}cs}: \emph{Families of canonically polarized
  varieties over surfaces}, Invent. Math. \textbf{172} (2008), no.~3, 657--682.
  {\sf\scriptsize DOI: 10.1007/s00222-008-0128-8}

\bibitem[KK08b]{KK08c}
{\sc S.~Kebekus and S.~J. Kov{\'a}cs}: \emph{The structure of surfaces and
  threefolds mapping to the moduli stack of canonically polarized varieties},
  preprint, {December 2008. arXiv:0812.2305}

\bibitem[KS06]{KS06}
{\sc S.~Kebekus and L.~{Sol{\'a} Conde}}: \emph{Existence of rational curves on
  algebraic varieties, minimal rational tangents, and applications}, Global
  aspects of complex geometry, Springer, Berlin, 2006, pp.~359--416.
  {\sf\scriptsize MR2264116}

\bibitem[{\hbox{K}}{\hbox{Mc}}99]{KMcK}
{\sc S.~{\hbox{K}}eel and J.~{\hbox{Mc}}Kernan}: \emph{Rational curves on
  quasi-projective surfaces}, Mem. Amer. Math. Soc. \textbf{140} (1999),
  no.~669, viii+153. {\sf\scriptsize 99m:14068}

\bibitem[Kol07]{Kollar07}
{\sc J.~Koll{\'a}r}: \emph{Lectures on resolution of singularities}, Annals of
  Mathematics Studies, vol. 166, Princeton University Press, Princeton, NJ,
  2007. {\sf\scriptsize MR2289519}

\bibitem[{\hbox{K}}{\hbox{M}}98]{KM98}
{\sc J.~{\hbox{K}}oll{\'a}r and S.~{\hbox{M}}ori}: \emph{Birational geometry of
  algebraic varieties}, Cambridge Tracts in Mathematics, vol. 134, Cambridge
  University Press, Cambridge, 1998, With the collaboration of C. H. Clemens
  and A. Corti, Translated from the 1998 Japanese original. {\sf\scriptsize
  2000b:14018}

\bibitem[Lan01]{MR1853454}
{\sc A.~Langer}: \emph{The {B}ogomolov-{M}iyaoka-{Y}au inequality for log
  canonical surfaces}, J. London Math. Soc. (2) \textbf{64} (2001), no.~2,
  327--343. {\sf\scriptsize MR1853454 (2002i:14009)}

\bibitem[Lan03]{MR1971155}
{\sc A.~Langer}: \emph{Logarithmic orbifold {E}uler numbers of surfaces with
  applications}, Proc. London Math. Soc. (3) \textbf{86} (2003), no.~2,
  358--396. {\sf\scriptsize MR1971155 (2004c:14069)}

\bibitem[Mum61]{MR0153682}
{\sc D.~Mumford}: \emph{The topology of normal singularities of an algebraic
  surface and a criterion for simplicity}, Inst. Hautes \'Etudes Sci. Publ.
  Math. (1961), no.~9, 5--22. {\sf\scriptsize MR0153682 (27 \#3643)}

\bibitem[Nam01]{Namikawa01}
{\sc Y.~Namikawa}: \emph{Extension of 2-forms and symplectic varieties}, J.
  Reine Angew. Math. \textbf{539} (2001), 123--147. {\sf\scriptsize MR1863856
  (2002i:32011)}

\bibitem[OSS80]{OSS}
{\sc C.~Okonek, M.~Schneider, and H.~Spindler}: \emph{Vector bundles on complex
  projective spaces}, Progress in Mathematics, vol.~3, Birkh\"auser Boston,
  Mass., 1980. {\sf\scriptsize MR561910 (81b:14001)}

\bibitem[Rei87]{Reid87}
{\sc M.~Reid}: \emph{Young person's guide to canonical singularities},
  Algebraic geometry, Bowdoin, 1985 (Brunswick, Maine, 1985), Proc. Sympos.
  Pure Math., vol.~46, Amer. Math. Soc., Providence, RI, 1987, pp.~345--414.
  {\sf\scriptsize MR927963 (89b:14016)}

\bibitem[Ste85]{Steenbrink85}
{\sc J.~H.~M. Steenbrink}: \emph{Vanishing theorems on singular spaces},
  Ast\'erisque (1985), no.~130, 330--341, Differential systems and
  singularities (Luminy, 1983). {\sf\scriptsize MR804061 (87j:14026)}

\bibitem[vSS85]{SS85}
{\sc D.~van Straten and J.~Steenbrink}: \emph{Extendability of holomorphic
  differential forms near isolated hypersurface singularities}, Abh. Math. Sem.
  Univ. Hamburg \textbf{55} (1985), 97--110. {\sf\scriptsize MR831521
  (87j:32025)}

\bibitem[Wah85]{Wahl85}
{\sc J.~M. Wahl}: \emph{A characterization of quasihomogeneous {G}orenstein
  surface singularities}, Compositio Math. \textbf{55} (1985), no.~3, 269--288.
  {\sf\scriptsize MR799816 (87e:32013)}

\end{thebibliography}

\providecommand{\bysame}{\leavevmode\hbox to3em{\hrulefill}\thinspace}
\providecommand{\MR}{\relax\ifhmode\unskip\space\fi MR}
\providecommand{\MRhref}[2]{%
  \href{http://www.ams.org/mathscinet-getitem?mr=#1}{#2}
}
\providecommand{\href}[2]{#2}

\end{document}